\documentclass[11pt]{article}
\usepackage{subcaption}
\usepackage{graphicx} 
\usepackage{float}
\usepackage{amssymb}
\usepackage{amsmath}
\usepackage{booktabs} 
\usepackage{amsthm}
\usepackage{xcolor}
\usepackage[final]{hyperref}
\usepackage{accents}
\usepackage{mathtools}
\usepackage[margin=0.5in]{geometry}

\makeatletter
\DeclareRobustCommand{\cev}[1]{%
  \mathpalette\do@cev{#1}%
}
\newcommand{\do@cev}[2]{%
  \fix@cev{#1}{+}%
  \reflectbox{$\m@th#1\vec{\reflectbox{$\fix@cev{#1}{-}\m@th#1#2\fix@cev{#1}{+}$}}$}%
  \fix@cev{#1}{-}%
}
\newcommand{\fix@cev}[2]{%
  \ifx#1\displaystyle
    \mkern#23mu
  \else
    \ifx#1\textstyle
      \mkern#23mu
    \else
      \ifx#1\scriptstyle
        \mkern#22mu
      \else
        \mkern#22mu
      \fi
    \fi
  \fi
}
\makeatother
\newtheorem{lemma}{Lemma}

\title{A PDE-Based Framework for Generative Modeling Beyond Classical Score-Based Diffusion}
\author{
Michael Herty\\
{\small Institute for Geometry and Practical Mathematics}\\
{\small RWTH Aachen University, Germany}\\
{\small Extraordinary Professor, Department of Mathematics and Applied Mathematics}\\
{\small University of Pretoria, South Africa} \\
{\small herty@igpm.rwth-aachen.de}
\and
Horacio Tettamanti\\
{\small Department of Mathematics ``F. Casorati''}\\
{\small University of Pavia, Italy}\\
{\small horacio.tettamanti01@universitadipavia.it}
}

\begin{document}
\maketitle
	\begin{abstract}

		We introduce an alternative generative framework based on a nonlinear modification of the classical Ornstein--Uhlenbeck dynamics. The proposed dynamics admits both a microscopic description through an interacting particle system and, in the mean-field limit, a macroscopic formulation given by a nonlinear Fokker--Planck equation with a superlinear drift term. We show that, for suitable choices of the model parameters and sufficiently large initial mass, the forward dynamics exhibits condensation phenomena by proving the loss of $L^2$ regularity of the solution in finite time. Building upon this formulation, we derive a stabilized reverse-time partial differential equation that reconstructs the initial distribution from the asymptotic state of the forward dynamics, thereby extending the generative paradigm beyond the classical score-based framework. Furthermore, we introduce numerical discretizations of both the forward and reverse processes that accurately capture the asymptotic behavior of the continuous model while successfully reconstructing the initial distribution. Numerical experiments in one and two spatial dimensions validate the proposed methodology and illustrate its application to density filtering through successive iterations of the generative process.

	\end{abstract}

\section{Introduction} \label{sect:intro}

In recent years, generative diffusion models have emerged as a powerful class of generative models, demonstrating remarkable capabilities across various domains such as computer vision, natural language processing, multi-modal modeling, among many others~\cite{BVRKB,DN,GLFK,HSGCNF,KPHZC,LTGLH,SCCLHSFN}. Several approaches have been developed, with score-based generative models~\cite{SSKKEP} and denoising diffusion probabilistic models~\cite{HJA} emerging as two of the most influential methodologies. We do not aim to review all existing work on diffusion models but follow the view point of \cite{CJLLM} and we refer to this publication for a more detailed referee of existing approaches. Furthermore, we point the interest reader to following references for more detailed review and novel mathematical approaches recently developed~\cite{D,ND,SWMG,SE,SE2,T,W}.

In general, we may state, the goal of a generative model is to produce new samples that are statistically consistent with a given dataset. Formally, let $\{x_i\}_{i=1}^N$ be samples drawn from an underlying distribution $f_0 \in \mathcal{P}(\mathbb{R}^d)$, where $d$ denotes the dimension of the ambient space. The objective is to generate a new set of samples $\{\tilde{x}_i\}_{i=1}^M$ that follow the same (or a similar) underlying distribution. A classical approach to nonparametric estimation of the underlying distribution is kernel density estimation, which builds an approximation of $f_0$ from the $N$ available samples~\cite{B,FRSST,STC}. While effective in low-dimensional settings, this method suffers from the curse of dimensionality, limiting its applicability in high dimensions. To overcome this curse of dimensionality, generative models have emerged as a powerful class of methods for learning and sampling from high-dimensional probability distributions. The general idea underlying these models consists of two complementary stages: a forward process and a backward process. During the forward step, the original data are progressively corrupted by noise, transforming the underlying distribution into a simpler reference distribution, typically chosen to be Gaussian. The backward step aims at reversing this transformation, thereby reconstructing the original distribution from the reference one. As we shall see, the information required to perform this reconstruction is encoded in a quantity known as the \textit{score function}, which plays a fundamental role in correctly steering the reverse dynamics toward the target distribution. Most generative models of this type are naturally formulated in terms of stochastic differential equations (SDEs). In this work, however, we focus on their emerging description at the PDE level. This perspective provides a complementary framework for the analysis of generative models and opens the possibility of employing analytical tools to gain a deeper understanding of such dynamics. To this end, we begin by briefly reviewing the classical generative process and its associated PDE formulation as introduced in~\cite{CJLLM}.

\textbf{A general review from a PDE perspective.} We begin our short review by describing the classical generative process and introducing the associated PDE formulation. To begin with, we consider the following stochastic differential equation (SDE) in $\mathbb{R}^d$ associated with the forward generative process know as the Ornstein--Uhlenbeck process (or Langevin Equation):
\begin{equation} \label{eq:ou}
    dX_t = -X_t\, dt + \sqrt{2}\, dB_t, \quad X_0 \sim \rho_0,
\end{equation}
where $\{B_t\}_{t\ge 0}$ is a $d$-dimensional Brownian motion and $\rho_0$ is the initial distribution of the particle system. An explicit solution for such system can be obtained by applying the integrating factor method, which yields
\begin{equation}
    X_t = e^{-t}X_0 + \sqrt{2}e^{-t}\int_0^t e^{z}\, dB_z
\end{equation}
from where we can notice that the distribution of $X_t$ is normally distributed as $\mathcal{N}\sim(e^{-t} X_0,\sigma^2_t Id)$ where $\sigma^2_t = 1 - e^{-2t}$. The asymptotic distribution of the forward process is therefore characterize by a standard normal distribution $\mathcal{N}(0,Id)$. Such process naturally induces an operator known as the Kolmogorov backward equation. From a different perspective, the same description may be re obtained by considering the mean-field limit of a system of $N$ interacting particles, which leads to the following Fokker--Planck equation for the evolution of the density $\rho(x,t)$ of the process \cite{Toscani}:
\begin{equation}
    \partial_t \rho = \nabla_x \cdot (x\rho) + \Delta \rho, \quad \rho(x,0) = \rho_0(x).
\end{equation}
In order to retrieve the original distribution $\rho_0$ from the forward process, we need to consider the backward-in-time dynamics. To do so, we introduce a new distribution $q(\mathbf{x},s)$ with $s = T - s$ such that $q(\mathbf{x},0) = \rho(\mathbf{x},T)$. We therefore obtain that the PDE associated with the backward process is given by

\begin{equation}
    \partial_s q
    =
    -\nabla_x\cdot(xq)
    -
    \Delta q,
\end{equation}
with $q(\cdot,0)=\rho(\cdot,T)$. Owing to the presence of the anti-diffusion operator, the above equation is ill-posed. A possible stabilization consists in decomposing the anti-diffusion term according to
\begin{equation}
    \Delta\phi
    =
    2\nabla\cdot(\nabla\phi)
    -
    \Delta\phi,
\end{equation}
and replacing the first occurrence of $\nabla\phi$ by the gradient of the known forward solution,
\[
\nabla q(\mathbf{x},s)
=
\nabla\rho(\mathbf{x},T-s).
\]
This yields the stabilized reverse process
\begin{equation}
    \partial_s q
    =
    -\nabla_x\cdot
    \Big(
    [-x-2\nabla\log\rho(\cdot,T-s)]\,q
    \Big)
    +
    \Delta q,
\end{equation}
provided that $q(\mathbf{x},s)=\rho(\mathbf{x},T-s)$. The above construction is not unique. For instance, one may instead consider the alternative decomposition
\begin{equation}
    \Delta\phi
    =
    2\Delta\phi
    -
    \Delta\phi,
\end{equation}
which leads to the reverse-time equation
\begin{equation}
    \partial_s q
    =
    \nabla_x\cdot
    \left(
    -xq
    -
    2\nabla\rho(\cdot,T-s)
    \right)
    +
    \Delta q.
\end{equation}
In this formulation, the information of the forward dynamics is encoded through the quantity
\[
-2\nabla\rho(\cdot,T-s),
\]
which may be interpreted as an alternative score function. More generally, these formulations illustrate that the generative process can be described entirely at the PDE level through a pair of forward and reverse evolution equations. The role of the score function is to encode the information accumulated during the forward dynamics and to steer the reverse evolution towards the reconstruction of the initial distribution. In practical implementations of score-based generative models, this quantity is typically approximated by a neural network. From an application standpoint, in order to apply such generative model we need to obtain a suitable SDE associated with the backward process. Following the work introduced in~\cite{CJLLM}, the SDE associated with the backward process is given by
\begin{equation}\label{ref:back_sde}
    d\vec{X}_s = \vec{X}_s\, ds + 2 \nabla \log \rho(\vec{X}_s,T-s)\, ds + \sqrt{2}\, d\vec{B}_s.
\end{equation}
A more explicit expression for the score function can be obtained by considering the conditional expectation of the initial position $X_0$ given that the process is at position $X_t = x$ at time $t$, namely
\begin{equation}\label{eq:x0}
    \vec{\mathbf{x}}_0(x,t) = \mathbb{E}[X_0 \mid X_t = x].
\end{equation}
By a direct application of Bayes' rule, together with the fact that the transition probabilities follow a  Gaussian distribution—due to the underlying Ornstein--Uhlenbeck dynamics—we obtain the following result. We refer the reader to~\cite{CJLLM} for a detailed proof.

\begin{lemma}
Let $X_0 \sim \rho_0$ be the initial distribution. Then
\begin{equation}
    \vec{\mathbf{x}}_0(x,t) =
    \frac{\displaystyle \int_{\mathbb{R}^d} y \, \rho_0(y)\, p(x,t \mid y,0)\, dy}
         {\displaystyle \rho(x,t)}.
\end{equation}
Moreover, the score function can be written as
\begin{equation}
    \nabla \log \rho(x,t) = \frac{e^{-t}\,\vec{\mathbf{x}}_0(x,t) - x}{1 - e^{-2t}}.
\end{equation}

\end{lemma}
Given the previous Lemma the reverse diffusion dynamics at the SDE level can be rewritten as
\begin{equation}
    d\vec{X}_s = \vec{X}_s\, ds + 2 \frac{\alpha_t\,\vec{\mathbf{x}}_0(\vec{X}_s,s) - \vec{X}_s}{\beta_t}\, ds + \sqrt{2}\, d\vec{B}_s.
\end{equation}
with $t = T - s$, $\alpha_t = e^{-t}$ and $\beta_t = 1 - e^{-2t}$. The drift term of the reverse SDE naturally decomposes into two parts: the forward drift and a score-based correction term. A straight way of deriving a numerical scheme for computing the update of $\vec{X}_s$ is to consider an Euler--Maruyama discretization of the SDE, which reads

\begin{equation}
    \vec{X}_{s+\Delta s} = \vec{X}_s +  \Big(\vec{X}_s + 2 \frac{\alpha_t\,\vec{\mathbf{x}}_0(\vec{X}_{s_n},t_n) - \vec{X}_s}{\beta_t}\Big)\, \Delta s + \sqrt{\Delta s_n}\, \epsilon_n
\end{equation}
with $t_n = T - s_n$ and $\epsilon_n \sim \mathcal{N}(0,I)$. From Equation~\eqref{eq:x0} we observe that $\vec{\mathbf{x}}_0$ depends on the entire distribution $\rho(x,t)$, making the update of $\vec{X}_s$ inherently nonlinear and dependent on the current state of the system. By assuming a discretization by an Euler-Maruyama scheme we are assuming that the whole scheme is freezed in a single time step $[s_n, s_{n+1}]$. A more accurate method is to assume that only the quantity $\vec{\mathbf{x}}_0$ is frozen in the time step. Under this assumption it can be shown that a suitable time update for $\vec{X}_{s_{n+1}}$ is given by

\begin{equation}\label{eq:sde_back_upd}
\begin{aligned}
\vec{X}_{s_{n+1}} &= \mu_{n+1} + \sigma_{n+1}\,\epsilon_n, \\
\mu_{n+1} &= 
\frac{e^{-\Delta s_n}\big(1 - e^{-2t_{n+1}}\big)}{1 - e^{-2t_n}}\,\vec{X}_n
+
\frac{e^{-t_{n+1}}\big(1 - e^{-2\Delta s_n}\big)}{1 - e^{-2t_n}}\,\langle \vec{x}_0 \rangle, \\
\sigma_{n+1}^2 &= 
\frac{\big(1 - e^{-2t_{n+1}}\big)\big(1 - e^{-2\Delta s_n}\big)}{1 - e^{-2t_n}}.
\end{aligned}
\end{equation}
where $|\Delta t_n| = |t_{n+1} - t_n| = |-s_{n+1} + s_n|$. This exact-in-time integration of the linearized reverse diffusion leads to a more accurate numerical scheme update of the backward solution compared to the standard Euler-Maruyama approach.

In this work, we show that the generative paradigm can be extended beyond the classical diffusion framework by constructing a generative process from a nonlinear modification of the Ornstein--Uhlenbeck dynamics. We formulate the proposed model at both the microscopic and macroscopic levels, deriving a nonlinear Fokker--Planck equation through well-established mean-field techniques for interacting particle systems. A distinctive feature of the resulting forward dynamics is the occurrence of finite-time condensation for suitable choices of the model parameters and sufficiently large initial mass. Consequently, the asymptotic state is attained within a finite time horizon, in contrast with classical diffusion-based generative models, whose equilibrium distribution is reached only asymptotically in time. Based on this formulation, we derive a stabilized reverse-time diffusion process and introduce a numerical discretization that accurately reproduces both the forward and reverse dynamics, successfully reconstructing the initial distribution from its asymptotic state. Finally, we illustrate the proposed framework through applications to density reconstruction and density filtering.

The remainder of the manuscript is organized as follows. In Section~2, we introduce the modified forward dynamics and derive the corresponding PDE formulation through the mean-field limit of a system of interacting particles. We then show that, for suitable choices of the model parameters, the forward dynamics may exhibit finite-time blow-up by proving the loss of $L^2$ regularity of the solution. Subsequently, following the same philosophy as in classical generative models, we derive the associated reverse-time diffusion process. In Section~3, we present a numerical discretization of the generative process in both one- and two-dimensional settings. In Section~4, we present several applications of the proposed framework, with particular emphasis on image filtering and image reconstruction. Finally, Section~5 is devoted to conclusions and perspectives for future research.

\section{Diffusion Model with Condensation-Type Dynamics}


As an alternative to the classical Ornstein--Uhlenbeck process, we consider the following nonlinear interacting particle system given by the following SDE:
\begin{equation}
dX_t^i
=
-
X_t^i
\left(
1+\left[f_{\varepsilon,N}(X_t^i)\right]^\alpha
\right)dt
+
\sqrt{2}\,dB_t^i,
\qquad
X_t^i(0)=x_i(0),
\end{equation}
where $\alpha > 0$ is a model parameter, $\{B_t^i\}_{i=1}^N$ are independent $d$-dimensional Wiener processes, and $f_{\varepsilon,N}$ denotes the regularized empirical density defined as
\begin{equation}
f_{\varepsilon,N}(x)
:=
\frac1N
\sum_{i=1}^N
\psi_\varepsilon(x-X_t^i),
\end{equation}
with $\psi_\varepsilon$ being a mollifier, which we choose as the Gaussian kernel
\begin{equation}
\psi_\varepsilon(y)
=
\frac1{(2\pi\varepsilon)^{d/2}}
\exp\left(-\frac{|y|^2}{2\varepsilon}\right).
\end{equation}
The modified drift term in the SDE introduces a nonlinear dependence on the empirical density of the particle system, whose effect is to enhance the attraction of particles towards regions of higher density. When the number of particles becomes large, it is natural to describe the collective dynamics through the evolution of the associated probability density. This can be achieved through the use of well established mean-field techniques~\cite{CFRT,CFTV,G,Toscani}. 

To this end, we consider a system of $N$ interacting particles $X_t^i\in\mathbb{R}^d$ where for any test function $\varphi\in C_c^\infty(\mathbb{R}^d)$ it follows from It\^o's formula that

\begin{equation}
\begin{aligned}
d\varphi(X_t^i)
&=
-
\nabla\varphi(X_t^i)\cdot
X_t^i
\left(
1+f_{\varepsilon,N}(X_t^i,t)^\alpha
\right)dt + \sqrt2\, \nabla\varphi(X_t^i)\cdot dB_t^i +
\Delta\varphi(X_t^i)\,dt.
\end{aligned}
\end{equation}
Therefore, summing over all the contributions of the particle system, we obtain
\begin{equation}
\begin{aligned}
\frac{1}{N}\sum_{i=1}^{N} d\varphi(X_t^i)
&=
-\frac{1}{N}\sum_{i=1}^{N}
\nabla\varphi(X_t^i)\cdot
X_t^i
\left(
1+f_{\varepsilon,N}(X_t^i,t)^\alpha
\right)\,dt
\\
&\quad
+
\frac{\sqrt{2}}{N}
\sum_{i=1}^{N}
\nabla\varphi(X_t^i)\cdot dB_t^i
\\
&\quad
+
\frac{1}{N}
\sum_{i=1}^{N}
\Delta\varphi(X_t^i)\,dt .
\end{aligned}
\end{equation}
Introducing the empirical measure associated with the particle system we formally obtain the weak formulation
\begin{equation}
\partial_t \langle f_N,\varphi\rangle
=
-
\left\langle
f_N,
x\left(1+f_{\varepsilon,N}(x,t)^\alpha\right)\cdot\nabla\varphi
\right\rangle
+
\left\langle
f_N,\Delta\varphi
\right\rangle .
\end{equation}
Integrating by parts yields
\begin{equation}
\left\langle
\partial_t f_N
-
\nabla_x\cdot
\left[
x\left(1+f_{\varepsilon,N}(x,t)^\alpha\right)f_N
\right]
-
\Delta f_N,
\varphi
\right\rangle
=0 
\end{equation}
which in strong form reads
\begin{equation}
\partial_t f_N
=
\nabla_x\cdot
\left[
x\left(1+f_{\varepsilon,N}(x,t)^\alpha\right)f_N
\right]
+
\Delta f_N .
\end{equation}
To establish compactness, we derive a uniform bound on the second moment. Applying It\^o's formula to $|X_t^i|^2$ gives
\begin{equation}
\frac{d}{dt}
\mathbb{E}\bigl[|X_t^i|^2\bigr]
=
-
2\,
\mathbb{E}
\left[
|X_t^i|^2
\bigl(1+f_{\varepsilon,N}(X_t^i,t)^\alpha\bigr)
\right]
+
2d.
\end{equation}
Since $1+f_{\varepsilon,N}^\alpha\ge1$, the second moment remains uniformly bounded in time. Consequently, by Prokhorov's theorem, the sequence $\{f_N\}_N$ is relatively compact in the weak-$*$ topology, and there exists a subsequence $(f_{N_\ell})_{\ell>0}$ such that
\[
f_{N_\ell}
\rightharpoonup
\rho,
\qquad
\ell\to\infty.
\]
Assuming that the regularized empirical density $f_{\varepsilon,N}$ converges strongly to $\rho$ as $N\to\infty$ and $\varepsilon\to0$, we formally obtain the nonlinear Fokker--Planck equation
\begin{equation}
\partial_t\rho
=
\nabla_x\cdot
\Bigl(
x(1+\rho^\alpha)\rho
\Bigr)
+
\Delta\rho,
\qquad
\rho(x,0)=\rho_0(x).
\label{eq:FP_f}
\end{equation}

Originally, this superlinear Fokker--Planck equation was introduced by Kaniadakis and Quarati as a semiclassical kinetic equation describing the dynamics of bosonic particles~\cite{KQ}. It was later employed in opinion dynamics to model conformity effects within a population~\cite{CDTZ}, and has since been adapted to the framework of Consensus-Based Optimization, where its condensation properties accelerate the concentration of particles towards the consensus state~\cite{FPZ}.

\subsection{Loss of \texorpdfstring{$L_2$}{L2} Regularity and Singularity Formation}

In the previous section, we formally derived the nonlinear Fokker--Planck equation
\eqref{eq:FP_f} as the mean-field limit of the interacting particle system. We now investigate one of its most distinctive qualitative features, namely the emergence of condensation phenomena. This behavior was originally introduced to describe the dynamics of bosonic particles, whose long-time evolution is characterized by the formation of a condensate. From the mathematical viewpoint, this corresponds to proving that the solution concentrates into a Dirac measure in finite time.

Throughout this section, we restrict our attention to the one-dimensional case and assume that $\alpha>2$. In this regime, provided that the initial mass is sufficiently large, the nonlinear drift induces finite-time condensation. By contrast, for $\alpha\le2$ it is known that the solution remains smooth for all times, and we refer the interested reader to~\cite{TZ} for a detailed analysis. To establish the formation of condensates, we begin by recalling that for every non-negative function $\phi\in L^2(\mathbb{R})$ we have

\begin{equation}
\begin{aligned}
     \int_{\mathbb{R}} \phi(x)\,dx &= \int_{R \leq |x|} \phi(x)\,dx + \int_{|x| < R} \phi(x)\,dx  \\
    &\leq \frac{1}{R^2} \int_{\mathbb{R}} x^2 \phi(x)\,dx + (2R)^{1/2}\Big(\int_{\mathbb{R}} \phi^2(x)\,dx\Big)^{1/2}.
\end{aligned}
\end{equation}    
Optimizing over the choice of the radius $R$ and re-writing terms, we obtain the following inequality
\begin{equation}\label{eq:moment_norm_rel}
    M_2[\phi]=\int_{\mathbb R}x^2\phi(x)\,dx
    \ge
    \left(\frac{1}{5}\right)^5
    \frac{(2\sqrt2)^4
    M_0^5}
    {\left(\displaystyle\int_{\mathbb R}\phi^2(x)\,dx\right)^2}.
\end{equation}
where $M_k = \int_{\mathbb R} x^k \phi(x)\,dx$ denotes the $k$-th moment. From this inequality we may observe that the second order moment is related to the $L^2$ norm of the solution. Hence, studying the evolution of the total energy can give us information about the regularity of the solution. We define the energy of the system as
\[
I(t):=\int_{\mathbb R}x^2f(x,t)\,dx
\]
which we differentiate in time and use \eqref{eq:FP_f} to obtain
\begin{equation}\label{eq:moment_identity}
\frac{d}{dt}I(t)
=
2M-2I(t)-2\int_{\mathbb R}x^2f^{\alpha+1}\,dx.
\end{equation}
Therefore, at each time the evolution of the energy depends on higher order moments of the distributions. We seek to find a bound, in terms of the energy itself, that allows us to quantify the time evolution of the energy. To do so we have the following result:

\begin{lemma}

For $\alpha>2$ we have that

\[
\int_{\mathbb R}x^2f^{\alpha+1}\,dx \ge \tilde{C}_{\alpha} M_0^{\frac{3\alpha}{2}} I(t)^{\frac{2-\alpha}{2}},
\]

with $\tilde{C}_{\alpha} =  
\Big(\frac{\alpha - 2}{2\alpha}\Big)^{\alpha} 2^{-\frac{3\alpha}{2}}$.
\end{lemma}

\begin{proof}
For any $R>0$ we may define the mass contained in the ball of radius $R$ as $M_R(t) = \int_{|x|\le R}f(x,t)\,dx$.

\begin{equation}
\begin{aligned}    
    M_R(t) = \int_{|x|\le R} f(x,t)\,dx = \int_{|x|\le R} \Big(|x|^{\frac{2}{\alpha+1}} f(x,t) \Big) |x|^{-\frac{2}{\alpha+1}}\,dx \leq \\
    \Big(\int_{|x|\le R} |x|^2 f^{\alpha+1}(x,t)\,dx\Big)^{\frac{1}{\alpha+1}} \Big(\int_{|x|\le R} |x|^{-\frac{2}{\alpha}}\,dx\Big)^{\frac{\alpha}{\alpha+1}}
\end{aligned}
\end{equation}
where we have applied Hölder's inequality with exponents $p = \alpha + 1$ and $q = \frac{\alpha + 1}{\alpha}$. Since $\alpha > 2$, we have that $\int_{|x|\le R} |x|^{-\frac{2}{\alpha}}\,dx = \frac{2\alpha}{\alpha - 2} R^{1 - \frac{2}{\alpha}}$. Therefore, we obtain the following bound 

\begin{equation}
\int_{\mathbb{R}} |x|^2 f^{\alpha + 1}(x,t)\,dx \geq C_{\alpha} M_R^{\alpha + 1} R^{-(\alpha - 2)}.
\end{equation}
Here,  we also used that $x^2 f^{\alpha + 1}$ is a positive quantity to extend the radius of integration to the whole real line and $C_\alpha$ is constant which depends solely on the value of $\alpha$. Next, we notice that setting $R = \sqrt{\frac{2 I(t)}{M}}$ we can write
\begin{equation}
\int_{|x| > R} f(x,t)\,dx \leq \frac{1}{R^2} \int_{\mathbb{R}} x^2 f(x,t)\,dx = \frac{I(t)}{R^2} = \frac{M_0}{2}.
\end{equation}
From which we obtain the following bound
\begin{equation}
    \int_{\mathbb{R}} |x|^2 f^{\alpha + 1}(x,t)\,dx \geq \tilde{C}_{\alpha} \left(\frac{M_0}{2}\right)^{\alpha + 1} \left(\frac{2 I(t)}{M_0}\right)^{-\frac{\alpha - 2}{2}}.
\end{equation}
where $\tilde{C}_{\alpha} = \Big(\frac{\alpha - 2}{2\alpha}\Big)^{\alpha} 2^{-\frac{3\alpha}{2}}$ and our proof is concluded.
\end{proof}

Making use of the previous lemma, we write the following bound for the time evolution of the energy

\begin{equation}
    I'(t) \le 2M - 2I(t) - 2\widetilde C_\alpha M^{\frac{3\alpha}{2}} I(t)^{-p}
\end{equation}
where we have defined $p:=\frac{\alpha-2}{2}>0$. Now we consider the function
\begin{equation}
    g(s):=s+As^{-p},
\end{equation}
where it can be easily verified that the minimum of such function is attained at $s_{min}=\left( pA \right)^{\frac{1}{p+1}}$ such that $g(s_{min}) = 2 K_\alpha M_0^3$ with $K_\alpha = (1+\frac{1}{p})(p C_\alpha)^{\frac{2}{\alpha}}$ and $C_\alpha = \Big( \frac{\alpha-2}{2\alpha}\Big)^{\alpha}$. Therefore, we have that if the inital mass is big enough, namely $M_0 > \frac{1}{\sqrt{K_\alpha}}$, then the second moment reaches zero in finite time. In particular, we may compute the characteristic time as 
\begin{equation}
    T^* = \frac{I(0)}{2K_\alpha M_0^3 - 2M_0}.
\end{equation}
This  proves that, provided the initial mass is sufficiently large, the second moment reaches zero in finite time. This in turn implies, from Equation \eqref{eq:moment_norm_rel}, that the $L^2$ norm of the solution blows up in finite time. A similar result can also be obtained in higher dimensional case,  where the critical condition is given for $\alpha > \frac{2}{d}$, where $d$ is the dimension of the ambient space. In conclusion, provided $\alpha > 2/d$ and the initial mass is sufficiently large, the solution of the nonlinear Fokker--Planck equation \eqref{eq:FP_f} loses $L^2$ regularity in finite time.

\subsection{Derivation of the Reverse-Time Diffusion Process}

Following the approach introduced in~\cite{CJLLM}, we now derive a reverse-time diffusion process associated with the nonlinear Fokker--Planck equation~\eqref{eq:FP_f}. The forward dynamics transports any initial distribution towards its stationary state, which, depending on the values of the model parameters, may be reached in finite time. The objective of the reverse dynamics is to reconstruct the initial distribution starting from the asymptotic state generated by the forward evolution. To this end, we introduce the time-reversed density
\[
q(\mathbf{x},s), \qquad s=T-t,
\]
with initial condition
\[
q(\mathbf{x},0)=\rho(\mathbf{x},T),
\]
where $\rho$ denotes the solution of the forward problem~\eqref{eq:FP_f}. A direct change of variables yields the formal backward equation
\begin{equation}
\partial_s q(\mathbf{x},s)
=
\nabla_x\cdot
\Bigl(
\mathcal{B}[\rho]\,q(\mathbf{x},s)
\Bigr)
-
\Delta q(\mathbf{x},s),
\end{equation}
where
\[
\mathcal{B}[\rho]
=
-\mathbf{x}\bigl(1+\rho^\alpha\bigr).
\]

The presence of the anti-diffusion term renders the above equation ill-posed. Following~\cite{CJLLM}, we introduce a stabilization by decomposing the diffusion operator as
\[
\Delta q
=
2\Delta q-\Delta q,
\]
and replacing the second occurrence of $q$ by the known solution of the forward problem,
\[
q(\mathbf{x},s)=\rho(\mathbf{x},T-s).
\]
This leads to the stabilized reverse-time diffusion process
\begin{equation}
\partial_s q
=
-\nabla_x\cdot
\Bigl(
\mathcal{B}_{\mathrm{rev}}[\rho]\,q
\Bigr)
+
\Delta q
-
2\Delta\rho(\mathbf{x},T-s),
\qquad
q(\mathbf{x},0)=\rho(\mathbf{x},T),
\end{equation}
where
\[
\mathcal{B}_{\mathrm{rev}}[\rho]
=
\mathbf{x}\bigl(1+\rho(\mathbf{x},T-s)^\alpha\bigr)
+
2\frac{\nabla\rho(\mathbf{x},T-s)}
{\rho(\mathbf{x},T-s)}.
\]
Equivalently, the reverse process can be written as
\begin{equation}
\partial_s q(\mathbf{x},s)
=
\nabla_x\cdot
\Bigl(
-\mathbf{x}
\bigl(
1+\rho(\mathbf{x},T-s)^\alpha
\bigr)
q(\mathbf{x},s)
\Bigr)
+
\Delta q(\mathbf{x},s)
-
2\Delta\rho(\mathbf{x},T-s).
\label{eq:FP_b}
\end{equation}
We remark that the above stabilization is not unique. Alternative regularization strategies can also be employed, and we refer the interested reader to~\cite{DM,KAAL} for further examples.

\section{Numerical Approximation of the Nonlinear Generative Process}

Following~\cite{CJLLM}, generative processes are typically formulated through stochastic differential equations, while the corresponding partial differential equations are primarily employed for the theoretical analysis of diffusion models. In this work, we adopt a complementary viewpoint by constructing the generative process directly at the PDE level. Our objective is to develop numerical discretizations that accurately capture both the forward and reverse dynamics described by equations~\eqref{eq:FP_f} and~\eqref{eq:FP_b}, respectively, while preserving their generative properties. For clarity of exposition, we first present the discretization in the one-dimensional setting. The extension to higher spatial dimensions is then obtained by means of dimensional splitting. Since the construction of the backward discretization follows the same procedure, we describe the extension only for the forward equation~\eqref{eq:FP_f}.

\subsection{Discretization of the Forward Process}

We first consider the numerical approximation of the forward equation. Our goal is to design a numerical scheme that is able to capture the asymptotic distribution of Equation~\eqref{eq:FP_f}. We begin by introducing a discretization of the spatial domain given by $x_i = -L + i\Delta x$, where $i = 0, \ldots, N_x$ and $\Delta x = \frac{2L}{N_x}$. Since the equation consists of a nonlinear transport term and a linear diffusion term, we discretize the two contributions separately. For the discretization of the diffusion operator we consider a standard second-order central finite difference scheme 

\begin{equation}
    \Delta_x f(x_i) \approx \frac{f(x_{i+1}) - 2f(x_i) + f(x_{i-1})}{\Delta x^2}
\end{equation}
For the non-linear drift we treat it as a conservation law with nonlinear velocity and we consider a finite volume scheme with an upwind discretization of the flux. We rewrite the drift as 

\begin{equation}
\begin{aligned}
    \nabla \cdot \Big[ x\,(1+\beta [ f(x,t) ]^{\alpha})\,&f(x,t) \Big] = -\nabla \cdot (v(x,t)f(x,t)) \\ 
    \text{with }  v(x,t) &= -x\,(1+\beta [ f(x,t) ]^{\alpha})
\end{aligned}
\end{equation}
which allows us to obtain the following discretization for the non-linear drift

\begin{equation}
\begin{aligned}
\partial_x \Big( x\,(1+\beta f^\alpha)\,f \Big)\Big|_{x_i}
&\approx
- \frac{F_{i+\frac{1}{2}} - F_{i-\frac{1}{2}}}{\Delta x}, \\[6pt]
F_{i+\frac{1}{2}} 
&= v_{i+\frac{1}{2}}^{\mathrm{eff}} \, f_{i+\frac{1}{2}}^{\mathrm{up}}, \\[6pt]
v_{i+\frac{1}{2}}^{\mathrm{eff}} 
&= - \frac{x_i + x_{i+1}}{2}
\left(1 + \beta \left[f_{i+\frac{1}{2}}^{\mathrm{up}}\right]^{\alpha}\right), \\[8pt]
f_{i+\frac{1}{2}}^{\mathrm{up}} 
&=
\begin{cases}
f_i,     & \text{if } v_{i+\frac{1}{2}}^{\mathrm{eff}} > 0, \\[4pt]
f_{i+1}, & \text{if } v_{i+\frac{1}{2}}^{\mathrm{eff}} < 0.
\end{cases}
\end{aligned}
\end{equation}
where the choice of $f_{i+\frac{1}{2}}^{\mathrm{up}}$ is such that the transported state is chosen according to the sign of the effective velocity $v_{i+\frac{1}{2}}^{\mathrm{eff}}$. By considering an explicit Euler scheme for the time discretization we obtain the following numerical scheme for the forward process:

\begin{equation}
    f_i^{n+1} = f_i^n - \Delta t \, \Big( -\frac{F_{i+\frac{1}{2}}^n - F_{i-\frac{1}{2}}^n}{\Delta x} +  \frac{f_{i+1}^n - 2f_i^n + f_{i-1}^n}{\Delta x^2} \Big)
\end{equation}
where the choice of the time step $\Delta t$ is such that the following CFL condition is satisfied:

\begin{equation}
    \Delta t \leq \min \Bigg\{ \frac{\Delta x^2}{2}, \frac{\Delta x}{\max_i |v_{i+\frac{1}{2}}^{\mathrm{eff}}|} \Bigg\}
\end{equation}
In this explicit formulation we observe that the size of the time step, $\Delta t$, becomes arbitrarily small in the critical case where the solution concentrates in a single point leading to the formation of a Dirac delta. To overcome this issue and relax the condition on the $\Delta t$ an implicit formulation could be considered. Although the explicit scheme becomes restrictive in the supercritical regime, we show in Section~\ref{sect:numerical_results} that, in the subcritical regime, it accurately reproduces the forward and reverse generative dynamics.

\subsection{Discretization of the  Backward Process}

We now turn to the numerical approximation of the backward stabilized problem. 
Given the forward solution $\rho(x,t)$, we denote by $q(x,s)$ the backward density, 
with $s = T - t$, and impose the terminal condition
\begin{equation}
q(x,T) = \rho(x,T).
\end{equation}
As in the forward case, we discretize the spatial domain using the grid 
$x_i = -L + i\Delta x$, $i = 0, \ldots, N_x$. The backward equation is composed of 
a transport term, a diffusion term, and a stabilization contribution depending on 
the forward density. As before, we treat each term separately. We discretize the Laplacian using the same second-order central difference:
\begin{equation}
\Delta_x q(x_i) \approx \frac{q_{i+1} - 2q_i + q_{i-1}}{\Delta x^2}.
\end{equation}
The backward transport is defined using a velocity field frozen from the forward 
solution. More precisely, at time level $s_k$ we define
\begin{equation}
v_i^{k} = x_i\Big(1 + [\rho_i^{k+1}]^{\alpha}\Big),
\end{equation}
where $\rho_i^{k+1}$ denotes the forward solution at the corresponding time. 
The transport term is discretized in conservative form:
\begin{equation}
\partial_x (v q)\Big|_{x_i}
\approx
- \frac{F_{i+\frac{1}{2}}^{k} - F_{i-\frac{1}{2}}^{k}}{\Delta x}.
\end{equation}
where the numerical flux is defined as
\begin{equation}
F_{i+\frac{1}{2}}^{k} 
= v_{i+\frac{1}{2}}^{k} \, q_{i+\frac{1}{2}}^{\mathrm{adj},k},
\end{equation}
with
\begin{equation}
v_{i+\frac{1}{2}}^{k} = \frac{v_i^{k} + v_{i+1}^{k}}{2}.
\end{equation}
Since the backward equation evolves in reverse time, the transport term is discretized using the reverse-upwind (downwind) state given by
\begin{equation}
q_{i+\frac{1}{2}}^{\mathrm{adj},k}
=
\begin{cases}
q_{i+1}^{k}, & \text{if } v_{i+\frac{1}{2}}^{k} > 0, \\[4pt]
q_i^{k},     & \text{if } v_{i+\frac{1}{2}}^{k} < 0.
\end{cases}
\end{equation}
Using an explicit Euler scheme, the backward update from $s_{k+1}$ to $s_k$ is given by
\begin{equation}
q_i^{k+1}
=
q_i^{k}
+
\Delta t \left[
- \frac{F_{i+\frac{1}{2}}^{k} - F_{i-\frac{1}{2}}^{k}}{\Delta x}
+
\frac{q_{i+1}^{k} - 2q_i^{k} + q_{i-1}^{k}}{\Delta x^2}
-
2 \, \frac{\rho_{i+1}^{k} - 2\rho_i^{k} + \rho_{i-1}^{k}}{\Delta x^2}
\right].
\end{equation}
Again we have that the time step $\Delta t$ is chosen such that the CFL condition is satisfied:
\begin{equation}
\Delta t \leq \min \left\{ \frac{\Delta x^2}{2}, \frac{\Delta x}{\max_i |v_i^{k}|} \right\}.
\end{equation}

\subsection{Extension to the Multi-Dimensional Case}
We extend the previous scheme to the two-dimensional setting on a Cartesian grid 
$(x_i,y_j)$ with uniform spacings $\Delta x$ and $\Delta y$. The extension follows using dimensional splitting. The forward equation reads
\begin{equation}
\partial_t \rho = \nabla \cdot \Big( (x,y)(1+\rho^\alpha)\rho \Big) + \Delta \rho.
\end{equation}
The diffusion term is discretized using standard second-order central differences,
\begin{equation}
\Delta_h \rho(x_i,y_j)
\approx
\frac{\rho_{i+1,j}-2\rho_{i,j}+\rho_{i-1,j}}{\Delta x^2}
+
\frac{\rho_{i,j+1}-2\rho_{i,j}+\rho_{i,j-1}}{\Delta y^2}.
\end{equation}
The nonlinear drift is treated as a conservation law and discretized using a finite volume upwind scheme applied separately in each spatial direction. Denoting by 
$\mathbf{v}=(v_x,v_y)$ with
\begin{equation}
v_x = -x(1+\rho^\alpha), 
\qquad
v_y = -y(1+\rho^\alpha),
\end{equation}
the divergence is approximated as
\begin{equation}
\nabla \cdot (\mathbf{v}\rho)
\approx
\frac{F^x_{i+\frac{1}{2},j}-F^x_{i-\frac{1}{2},j}}{\Delta x}
+
\frac{F^y_{i,j+\frac{1}{2}}-F^y_{i,j-\frac{1}{2}}}{\Delta y},
\end{equation}
where the fluxes are defined by upwinding,
\begin{equation}
F^x_{i+\frac{1}{2},j}
=
v^x_{i+\frac{1}{2},j}\,\rho^{\mathrm{up}}_{i+\frac{1}{2},j},
\qquad
F^y_{i,j+\frac{1}{2}}
=
v^y_{i,j+\frac{1}{2}}\,\rho^{\mathrm{up}}_{i,j+\frac{1}{2}},
\end{equation}
with the upwind states selected according to the sign of the corresponding velocity components. Using an explicit Euler discretization in time, the forward update is obtained as
\begin{equation}
\rho_{i,j}^{n+1}
=
\rho_{i,j}^{n}
+
\Delta t \left[
\frac{F^x_{i+\frac{1}{2},j}-F^x_{i-\frac{1}{2},j}}{\Delta x}
+
\frac{F^y_{i,j+\frac{1}{2}}-F^y_{i,j-\frac{1}{2}}}{\Delta y}
+
\Delta_h \rho_{i,j}^{n}
\right].
\end{equation}
The backward problem is constructed analogously to the forward case and we omit the details for brevity. As in the one-dimensional case, the CFL condition is determined by both the nonlinear drift, through the associated velocity field, and the diffusion term. In the two-dimensional setting, the time step $\Delta t$ must therefore satisfy the following stability condition:

\begin{equation}
\Delta t \leq \min \left\{
\frac{\Delta x}{\max_{i,j}|v^x_{i,j}|},
\frac{\Delta y}{\max_{i,j}|v^y_{i,j}|},
\frac{1}{2\left(\frac{1}{\Delta x^2}+\frac{1}{\Delta y^2}\right)}
\right\}.
\end{equation}

\section{Numerical Results}\label{sect:numerical_results}

In this section, we present several numerical experiments illustrating the capabilities of the proposed generative framework in both one and two spatial dimensions. We first validate the proposed numerical scheme by verifying that the forward and reverse dynamics accurately reconstruct the initial distribution from its asymptotic state. We then investigate the iterative application of the generative process as a filtering mechanism, showing that successive forward and reverse evolutions progressively remove high-frequency oscillations while preserving the main features of the underlying distribution. Finally, we consider the problem of density reconstruction in the presence of missing data and show that the particle formulation is able to recover missing regions of the distribution by properly modifying the prior distribution.

\subsection{Validation of the Numerical Generative Process}

We begin by validating the proposed numerical discretization through two one-dimensional reconstruction tests. First we consider two one-dimensional initial distributions given by

\begin{equation}
\begin{aligned}
f_0^{(1)}(x) &= C \Big( 2 e^{-30(x+3)^2} + 0.5 e^{-30(x-3)^2} + 0.5 e^{-30(x-1)^2} + 2 e^{-30(x+1)^2}\Big) \\[10pt]
f_0^{(2)}(x) &= C \Big( e^{-30(x+0.1)^2} + e^{-30(x-0.3)^2} + e^{-30(x-2.5)^2} + e^{-30(x+1.5)^2} \Big)
\end{aligned}
\end{equation}

\begin{figure}
    \centering

    \begin{subfigure}[t]{0.4\textwidth}
        \centering
        \includegraphics[width=\textwidth]{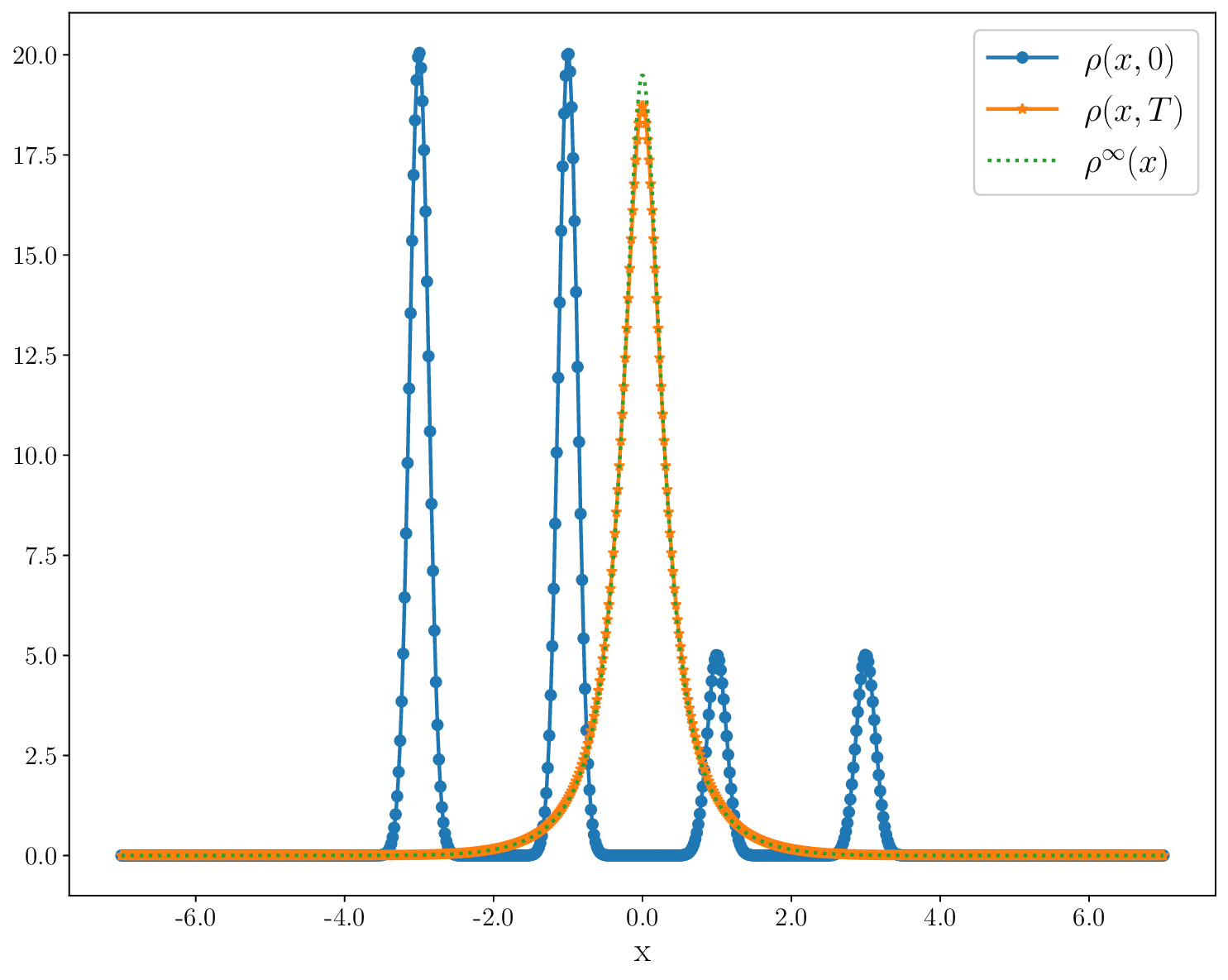}
        \caption{}
    \end{subfigure}
    \begin{subfigure}[t]{0.4\textwidth}
        \centering
        \includegraphics[width=\textwidth]{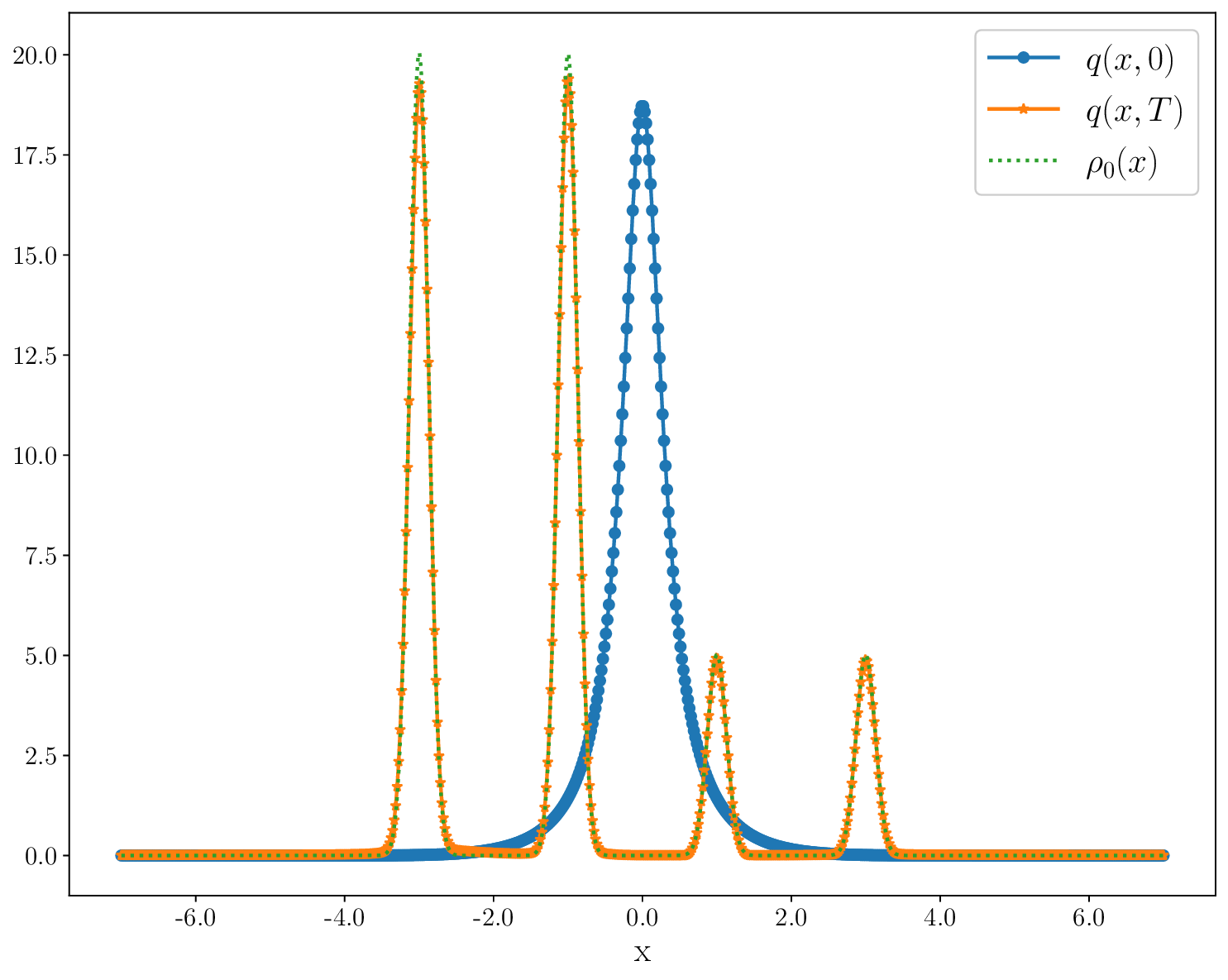}
        \caption{}
    \end{subfigure}

    \caption{
    Generative process of the numerical scheme with $\rho(\cdot,0)\sim f_0^{(1)}$, final time $T_f=10$, and initial mass $\mu=\mu_c/6$, chosen to avoid restrictive CFL conditions.
    \textbf{Left:} Initial distribution, numerical approximation of the stationary state obtained from the forward evolution, and the corresponding analytical solution.
    \textbf{Right:} Reconstructed initial distribution obtained after the backward process, together with the initial condition of the backward dynamics, which coincides with the stationary state of the forward step.
    In both cases, the numerical scheme accurately captures the structure of the Bose--Einstein distribution and successfully reconstructs the initial distribution. \textbf{Forward Error:} 0.0303, \textbf{Backward Error:} 0.0405.
    }
    \label{fig:1d_1}
\end{figure}

\begin{figure}
    \centering

    \begin{subfigure}[t]{0.4\textwidth}
        \centering
        \includegraphics[width=\textwidth]{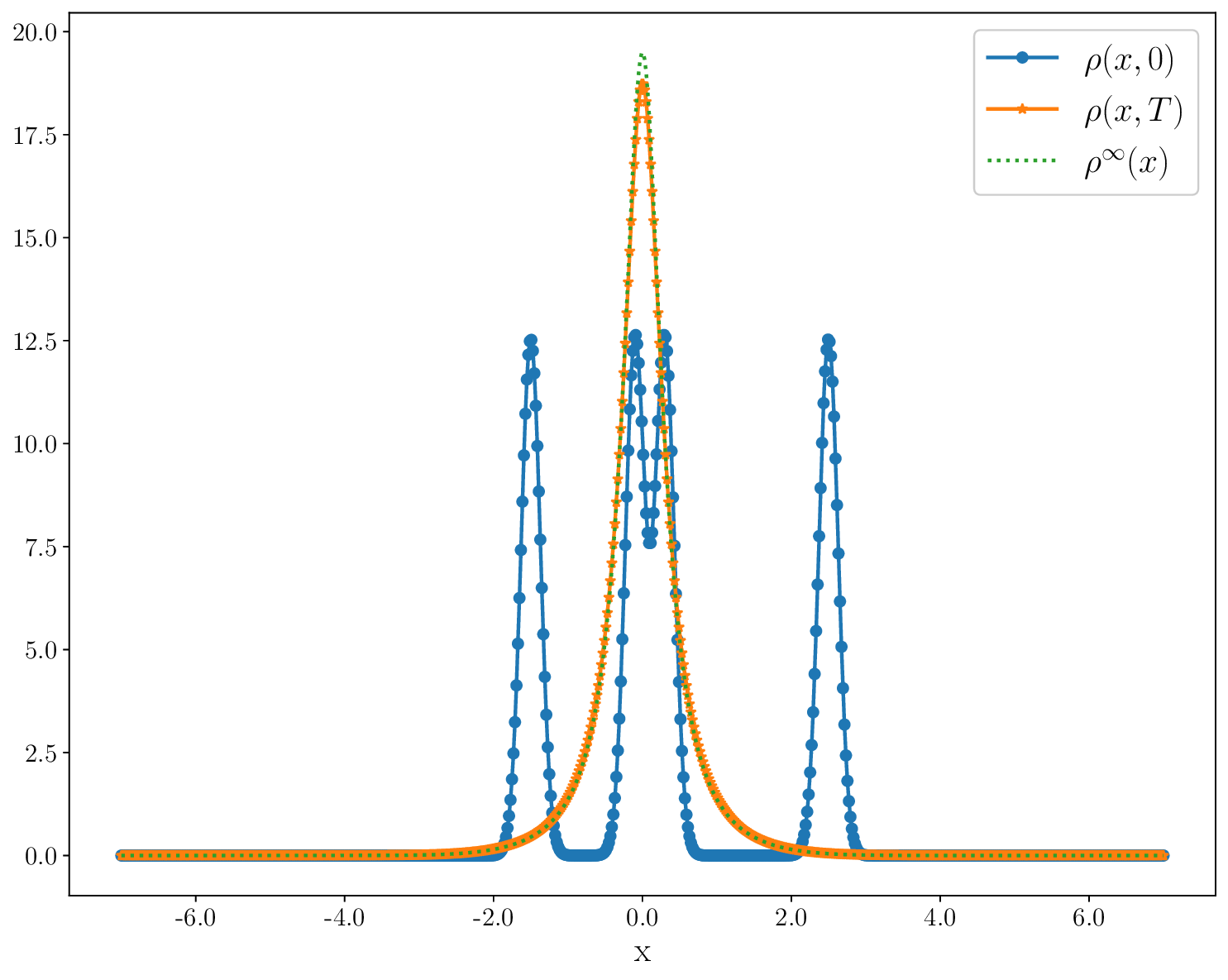}
        \caption{}
    \end{subfigure}
    \begin{subfigure}[t]{0.4\textwidth}
        \centering
        \includegraphics[width=\textwidth]{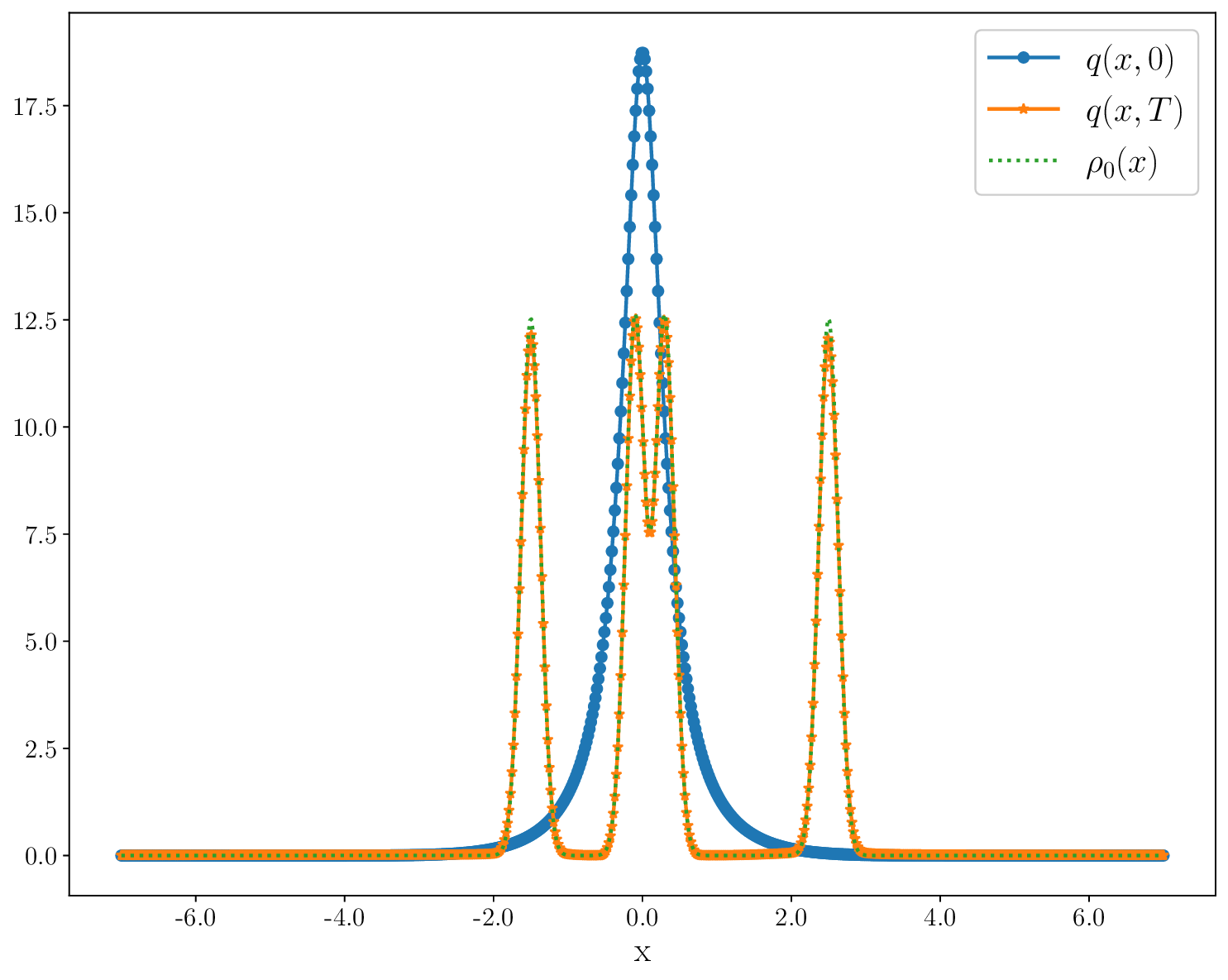}
        \caption{}
    \end{subfigure}

    \caption{
    Generative process of the numerical scheme with $\rho(\cdot,0)\sim f_0^{(2)}$, final time $T_f=10$, and initial mass $\mu=\mu_c/6$, chosen to avoid restrictive CFL conditions.
    \textbf{Left:} Initial distribution, numerical approximation of the stationary state obtained from the forward evolution, and the corresponding analytical solution.
    \textbf{Right:} Reconstructed initial distribution obtained after the backward process, together with the initial condition of the backward dynamics, which coincides with the stationary state of the forward step.
    In both cases, the numerical scheme accurately captures the structure of the Bose--Einstein distribution and successfully reconstructs the initial distribution. \textbf{Forward Error:} 0.0298, \textbf{Backward Error:} 0.0283.
    }
    \label{fig:1d_2}
\end{figure}

where $C$ is a normalization constant. We apply the scheme outlined in the previous section with $Nx=700$ in the domain $x \in [-L,L]$ with $L = 7$. As model parameter we set $\alpha = 1$ and as inital mass we chose $\mu = \mu_c / 60$ so that we work in the sub-critical regime and we avoid the formation of stiff CFL conditions. In Figure~\ref{fig:1d_1} and~\ref{fig:1d_2} we show the forward and backward evolution of the two initial distributions. In both cases, the backward dynamics accurately reconstruct the initial distribution from the asymptotic state, confirming that the numerical discretization preserves the generative properties of the continuous model. We next consider the same experiment in two spatial dimensions to assess the performance of the proposed framework beyond the one-dimensional setting. In this case we considered as initial distribution a mixture of four Gaussians given by

\begin{figure}
    \centering
    \begin{subfigure}{0.3\textwidth}
        \centering
        \includegraphics[width=\textwidth]{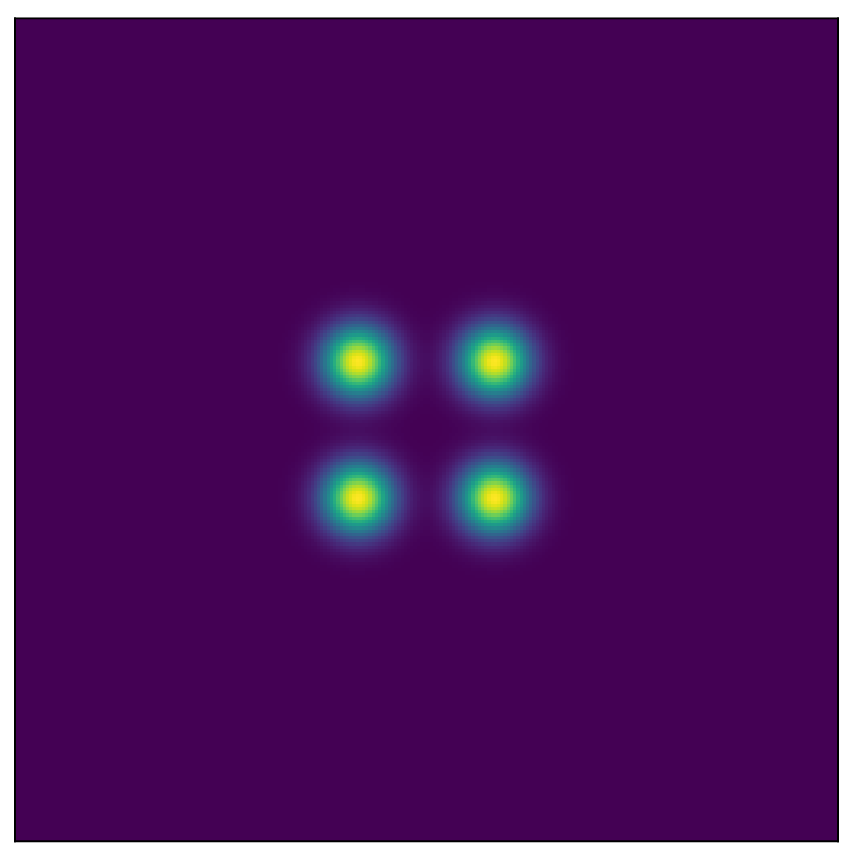}
        \caption{}
    \end{subfigure}

    \vspace{1cm}

    \begin{subfigure}{0.3\textwidth}
        \centering
        \includegraphics[width=\textwidth]{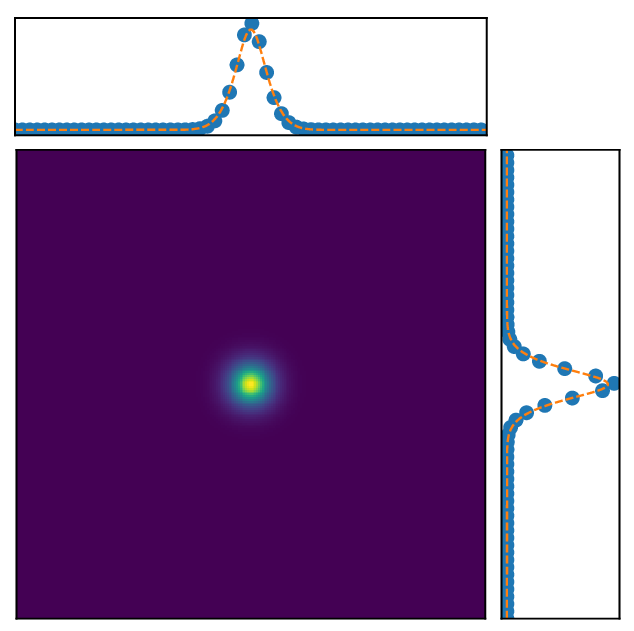}
        \caption{}
    \end{subfigure}
    \hspace{0.05cm}
    \begin{subfigure}{0.3\textwidth}
        \centering
        \includegraphics[width=\textwidth]{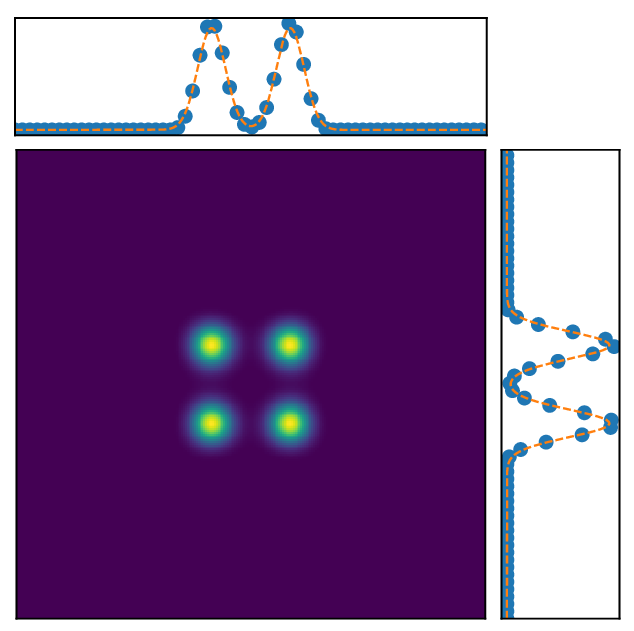}
        \caption{}
    \end{subfigure}

    \caption{
        Reconstruction of a two-dimensional distribution via forward and backward diffusion dynamics. 
        \textbf{Top (a):} Initial distribution. 
        \textbf{Bottom (b)--(c):} Results of the forward and backward processes, respectively. In each case, the reconstructed two-dimensional density is shown together with its marginal distributions along the spatial directions. The numerical marginals are compared against the corresponding analytical marginals, showing very good agreement.\textbf{Forward Error:} 0.0838, \textbf{Backward Error:} 0.0824.
        }
    \label{fig:2d}
\end{figure}

\begin{equation}
f_0(x,y) = \sum_{\boldsymbol{\mu}\in \mathcal{M}} 
\exp\!\left(-\| (x,y) - \boldsymbol{\mu} \|^2\right),
\quad 
\mathcal{M} = \{(2,2),(2,-2),(-2,2),(-2,-2)\}
\label{eq:2d_ini}
\end{equation}

where $(x,y) \in [-L,L]^2$ with $L = 12$ to ensure the domain was large enough so that we reduce the boundary effects. We used $Nx = Ny = 256$ grid points in both directions with $\alpha = 1$ with initial mass $\mu = \mu_c / 8$ again to ensure we are working in the subcritical regime. The results are shown in Figure~\ref{fig:2d} where we observe, first, the form of the initial distribution and, second, the forward and backward evolution of the density where we outline the marginals distributions and compare them with the analytical marginal distributions. This first set of set clearly demonstrates that the proposed stabilized reverse-time diffusion process and the introduced numerical scheme are able to correctly capture the generative process and reconstruct the initial distribution starting from the asymptotic distribution of the forward process.

\subsection{Generative Process as a Filtering Process}

In this test, we present an application of the proposed generative framework for filtering noisy distributions, which are typically reconstructed from scarce data. We also illustrate its relevance in the context of image filtering. We begin with a one-dimensional example by considering the initial distribution
\begin{equation}
\rho_0(x) = C \Big( e^{-10(x+1)^2} + e^{-10x^2} + e^{-10(x-1)^2}\Big),
\end{equation}
shown in Figure~\ref{fig:1d_filtering}(a). We sample this distribution using $N_p = 100$ particles and reconstruct it via a histogram-based approach with $N_x = 80$ bins over the domain $[-L,L]$, with $L=7$ and $\Delta x = \frac{2L}{N_x}$. This yields a noisy approximation of the initial distribution. By iteratively applying the generative process, we observe a progressive smoothing of the distribution, effectively acting as a filtering mechanism. The results are reported in Figure~\ref{fig:1d_filtering}, where $q_i(x,T)$ denotes the distribution obtained at the $i$-th iteration, for $i=1,2,3$. The model parameter is set to $\alpha = 1$, and normalization is enforced at each step. Successive applications of the forward and reverse dynamics progressively attenuate high-frequency oscillations while preserving the dominant features of the distribution.

\begin{figure}
    \centering

    \begin{subfigure}[b]{0.4\textwidth}
        \centering
        \includegraphics[width=\textwidth]{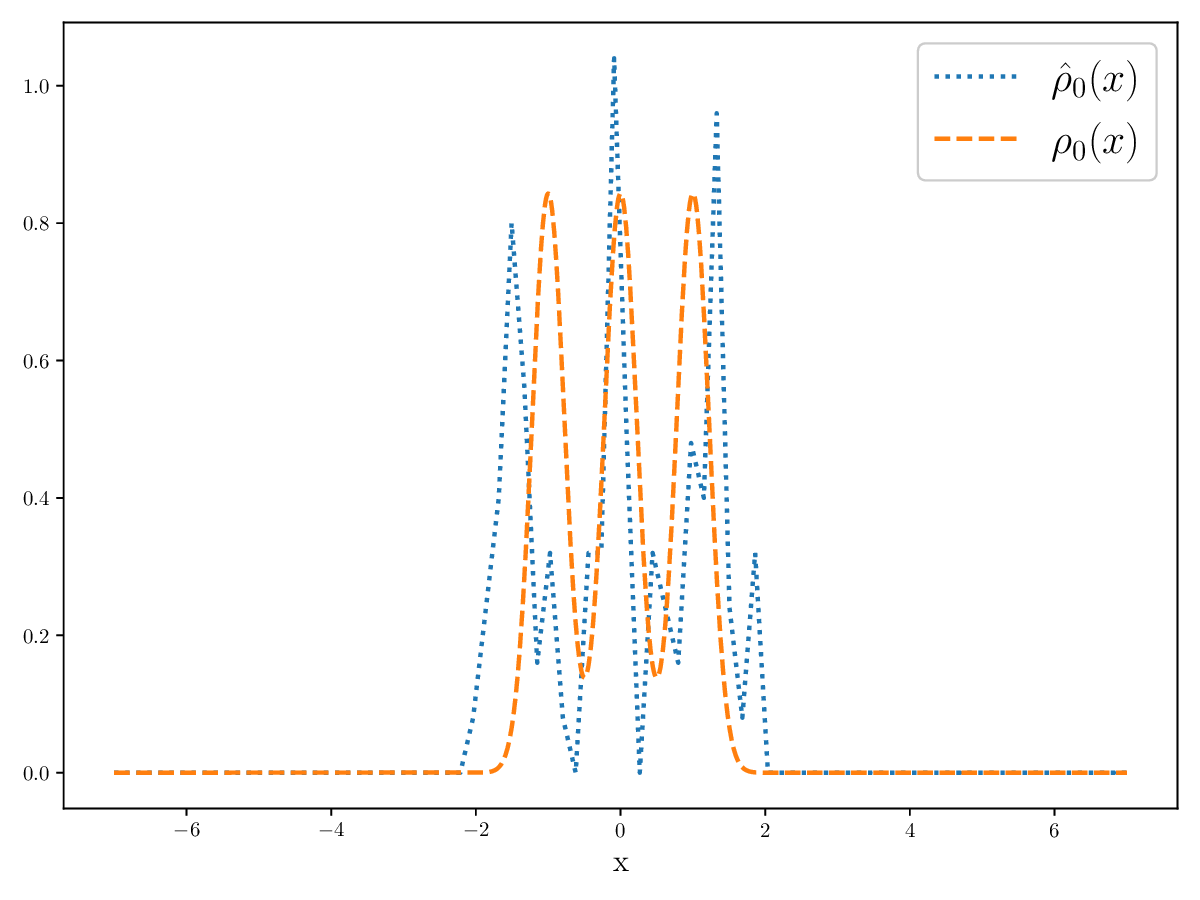}
        \caption{}
    \end{subfigure}
    \begin{subfigure}[b]{0.4\textwidth}
        \centering
        \includegraphics[width=\textwidth]{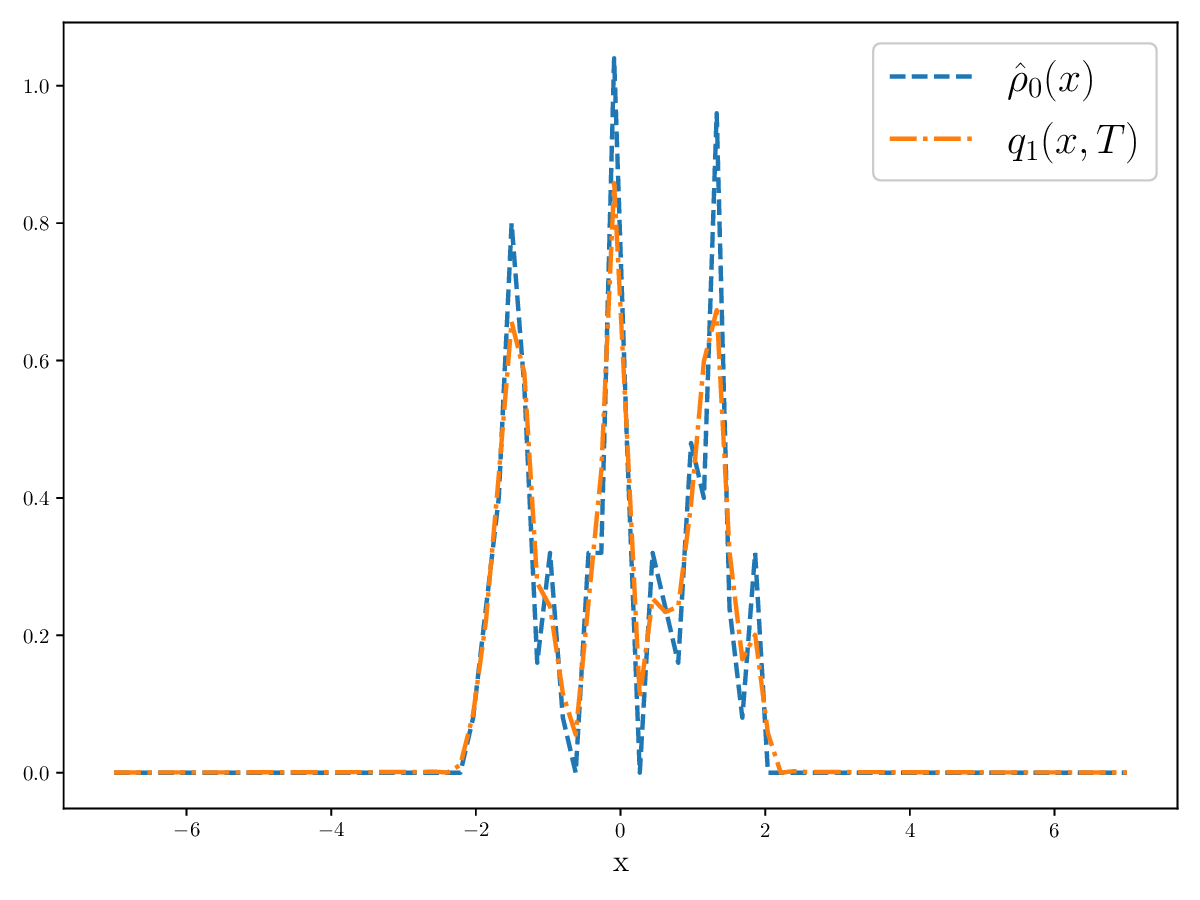}
        \caption{}
    \end{subfigure}

    \begin{subfigure}[b]{0.4\textwidth}
        \centering
        \includegraphics[width=\textwidth]{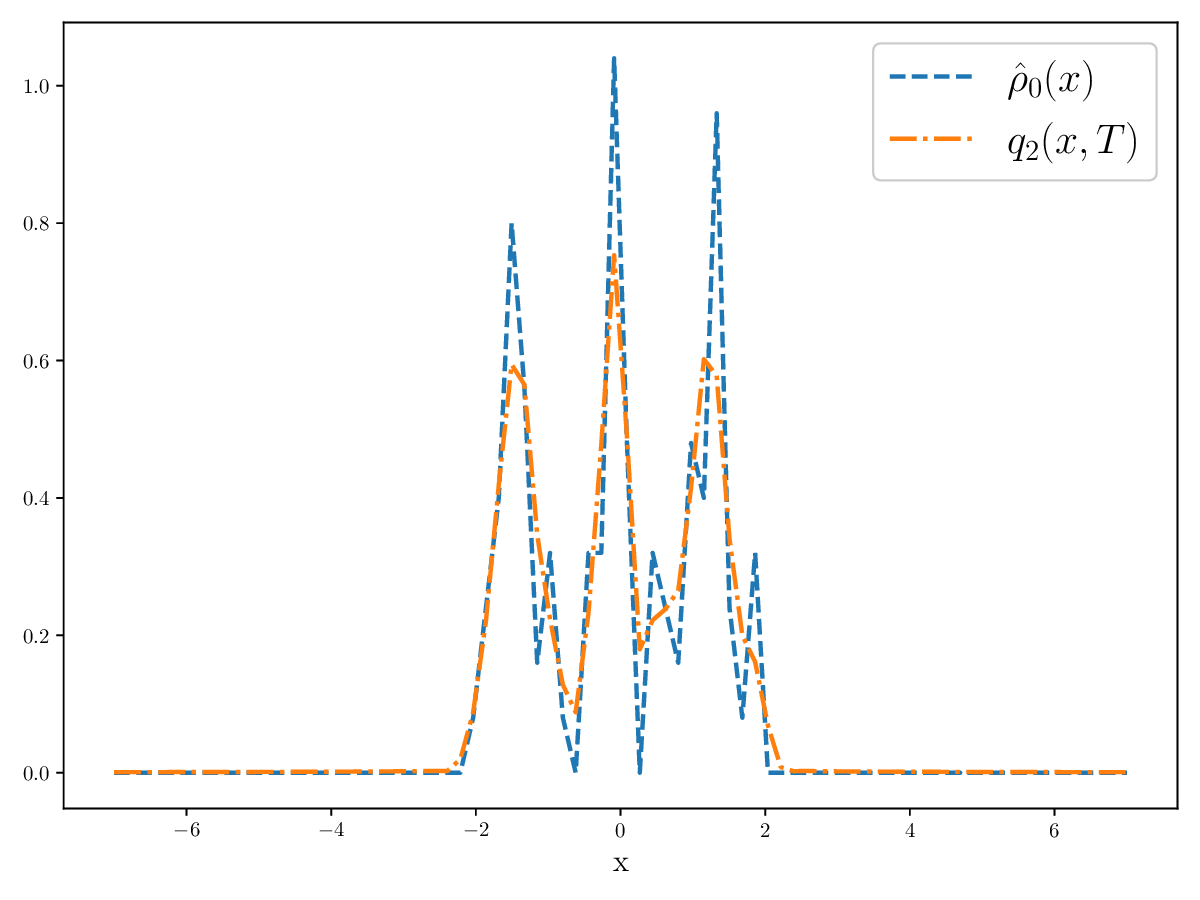}
        \caption{}
    \end{subfigure}
    \begin{subfigure}[b]{0.4\textwidth}
        \centering
        \includegraphics[width=\textwidth]{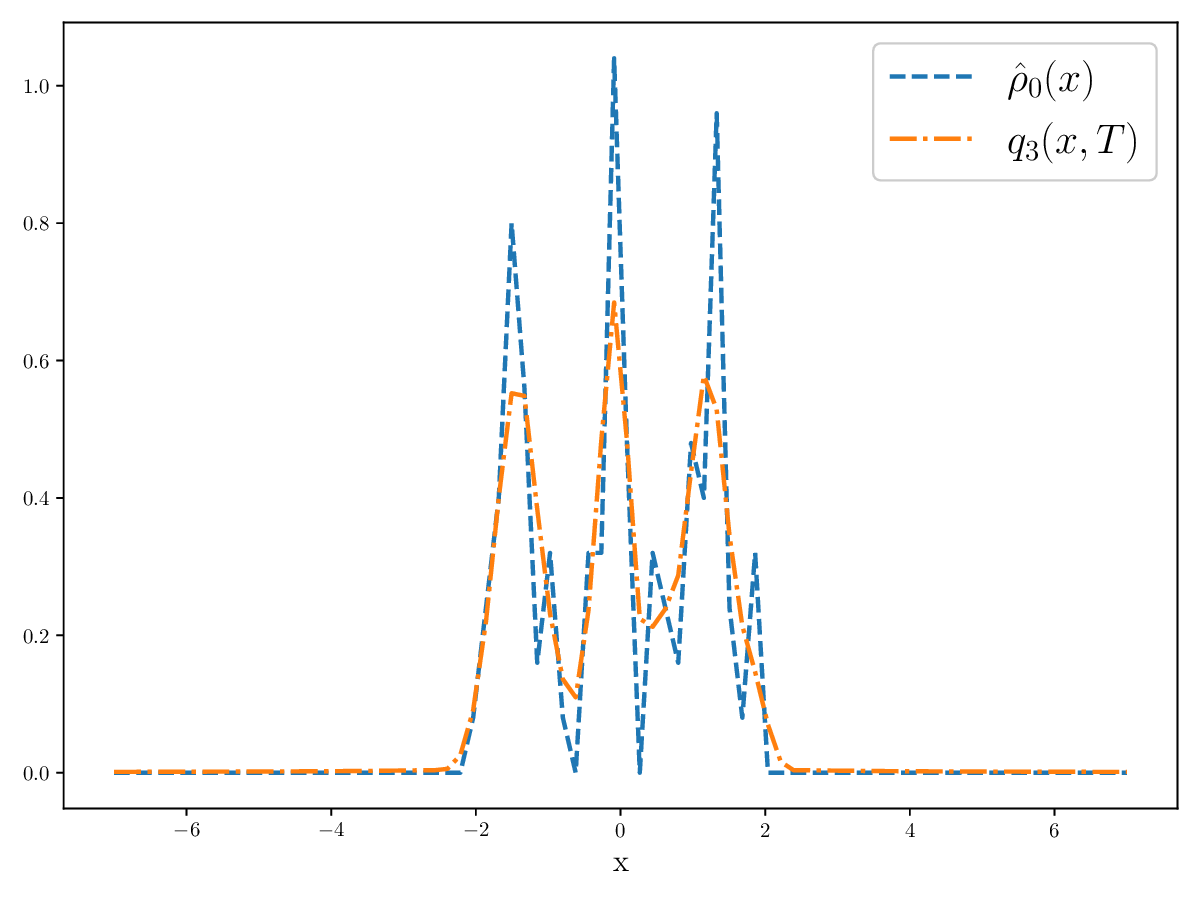} 
        \caption{}
    \end{subfigure}

    \caption{Smoothing effect of the generative process on a noisy distribution. In (a) we show the initial distribution and the one obtained when sampled from $N_p = 100$ particles while considering $80$ histograms. In (b), (c), and (d) we show the results of the generative process applied to the noisy distribution for different number of iterations in contrast with the original noisy distribution. It can be clearly observed that the generative process acts as a filtering mechanism, progressively smoothing the distribution and attenuating the peaks.}
    \label{fig:1d_filtering}
\end{figure}

We next extend the analysis to a two-dimensional setting, considering the initial distribution defined in Equation~\eqref{eq:2d_ini}. The spatial domain is $[-L,L]^2$ with $L=8$, discretized using $N_x = N_y = 128$ grid points. A noisy reconstruction is obtained from $N_p = 10^4$ particles using a histogram-based method. The results are shown in Figure~\ref{fig:2d_filtering}, where we compare the initial, reconstructed, and filtered distributions after a single iteration of the generative process. Owing to the larger number of neighboring cells involved in the evolution, the filtering effect is more pronounced in two dimensions, and a single iteration already provides a significant reduction of the sampling noise. We also report the marginal distributions and slices at $y = -1.35, 0.07, 1.49$, where the smoothing efffect is clearly visible through the attenuation of peaks.

\begin{figure}
    \centering

    \begin{subfigure}{\textwidth}
        \centering
        \includegraphics[width=0.8\textwidth]{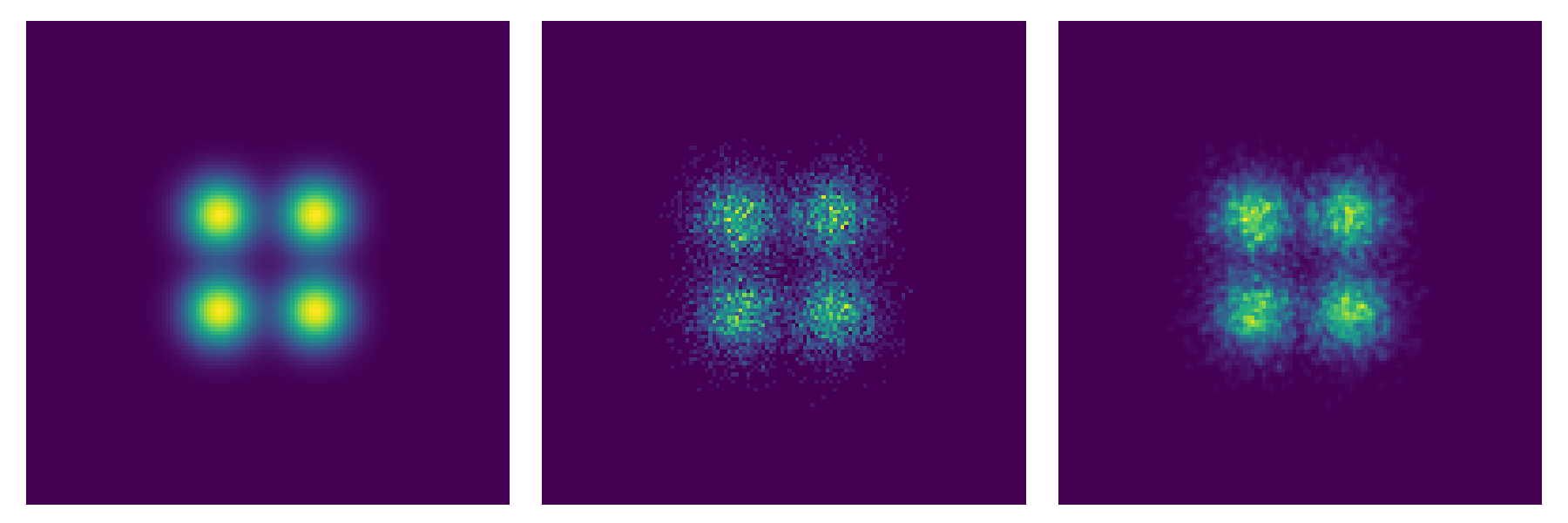}
        \caption{}
    \end{subfigure}

    \begin{subfigure}{\textwidth}
        \centering
        \includegraphics[width=\textwidth]{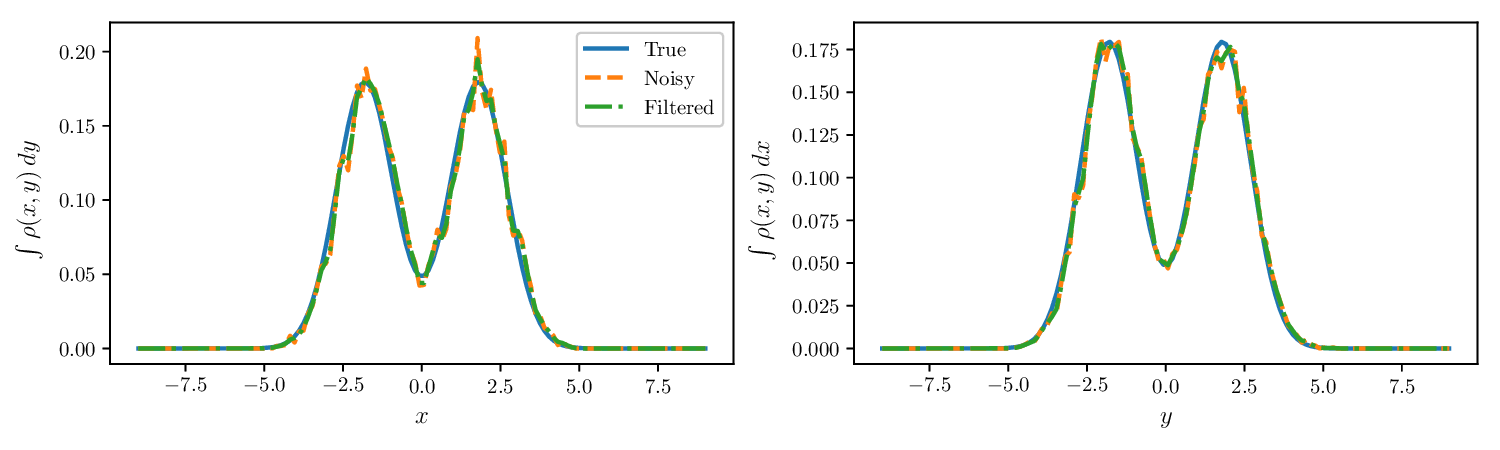}
        \caption{}
    \end{subfigure}

    \begin{subfigure}{\textwidth}
        \centering
        \includegraphics[width=0.8\textwidth]{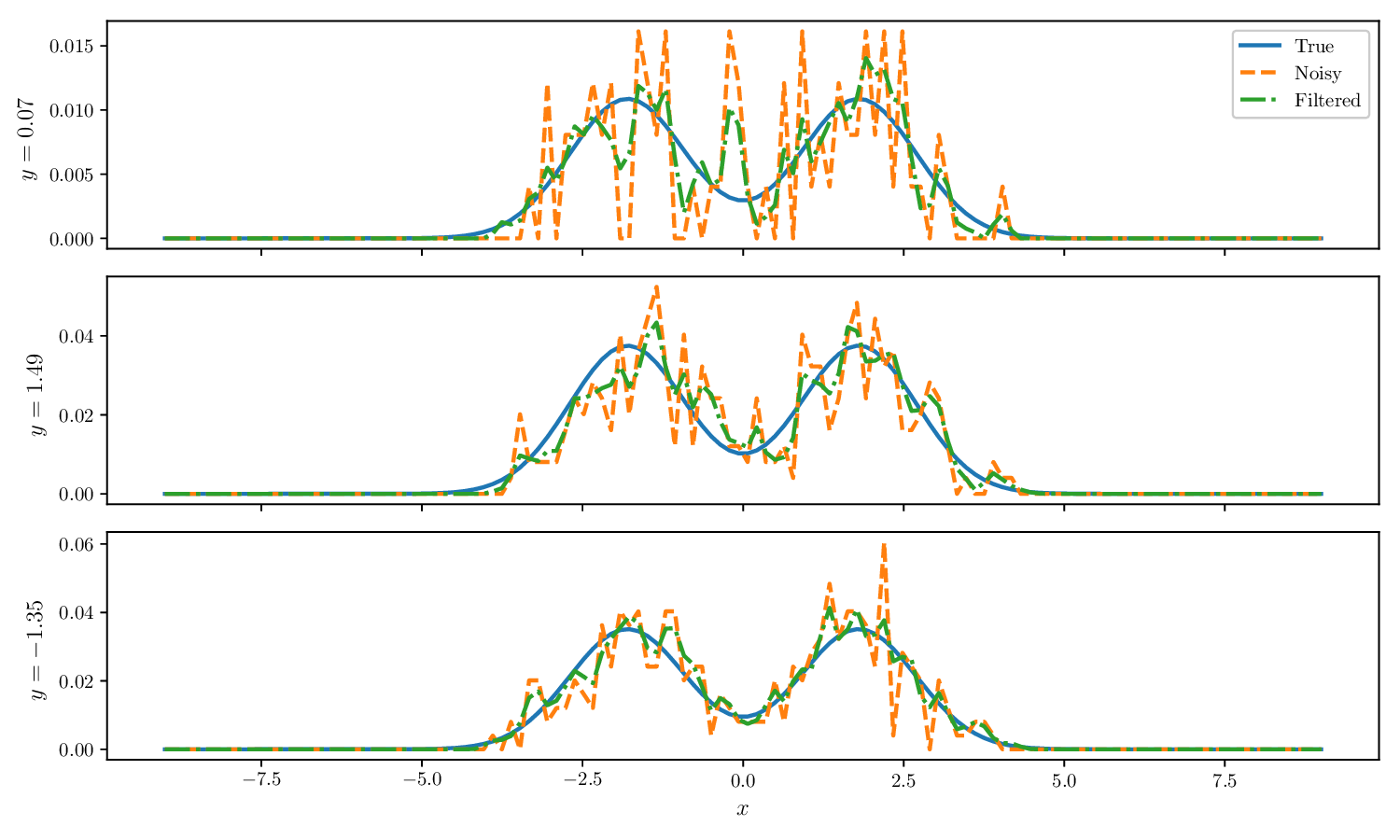}
        \caption{}
    \end{subfigure}

    \caption{
        Generative process acting as a filtering mechanism on a noisy two-dimensional distribution. 
        (a) Initial distribution, its noisy reconstruction from $N_p = 10^4$ particles using a $128 \times 128$ histogram, and the filtered result after one iteration. 
        (b) Marginal distributions along the spatial directions, showing clear smoothing through peak attenuation. 
        (c) Slices at $y = -1.35, 0.07, 1.49$, further illustrating the filtering effect.
        }
    \label{fig:2d_filtering}
\end{figure}

\begin{figure}
    \centering

    \begin{subfigure}{\textwidth}
        \centering
        \includegraphics[width=0.8\textwidth]{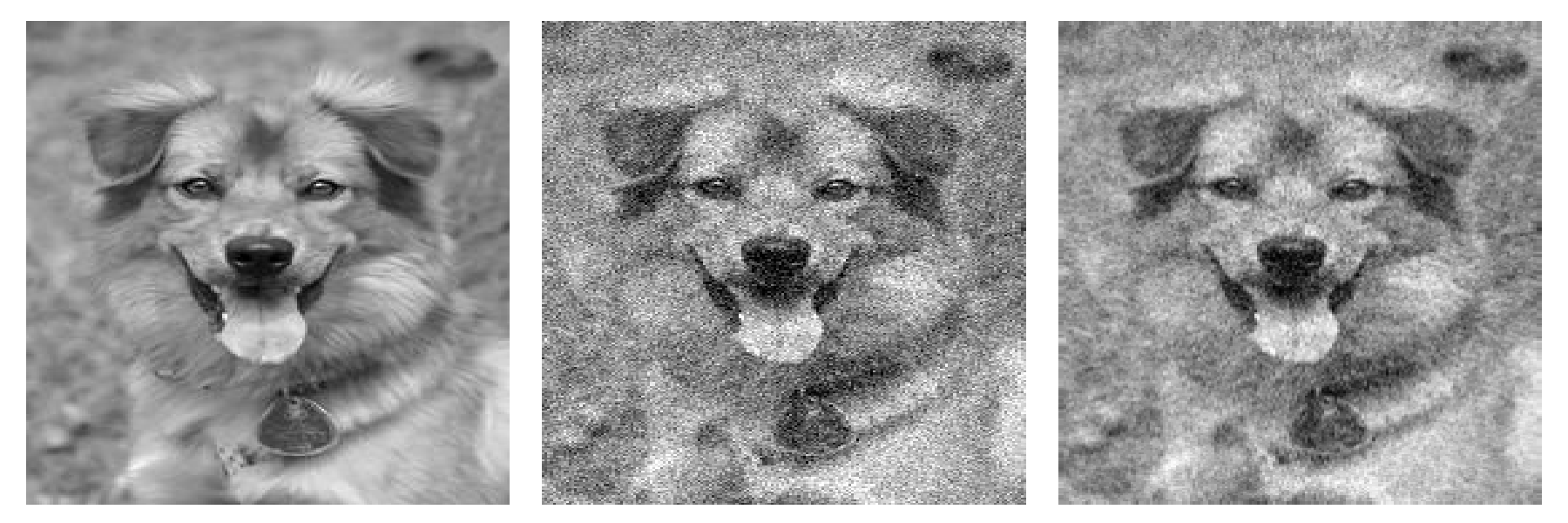}
        \caption{}
    \end{subfigure}

    \begin{subfigure}{\textwidth}
        \centering
        \includegraphics[width=0.8\textwidth]{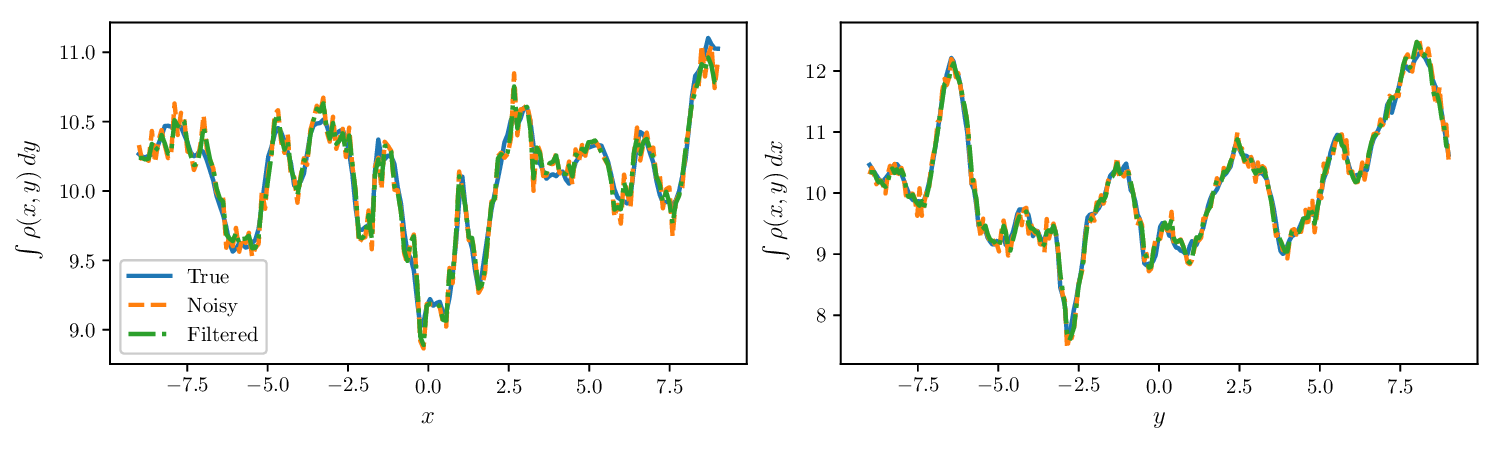}
        \caption{}
    \end{subfigure}

    \begin{subfigure}{\textwidth}
        \centering
        \includegraphics[width=0.8\textwidth]{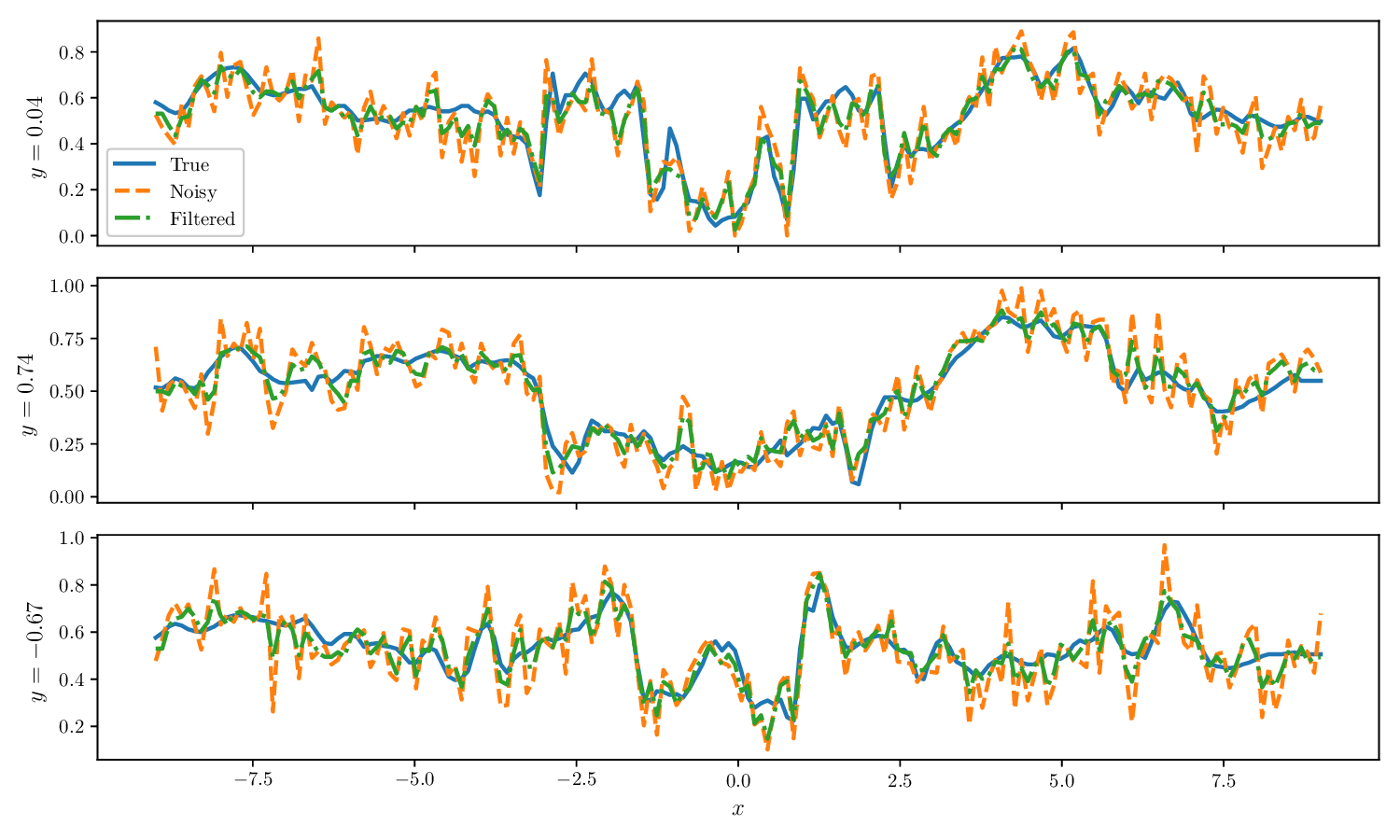}
        \caption{}
    \end{subfigure}

    \caption{Generative application for Image smoothing. In (a) we show the original image, a noisy image after applying a Gaussian filter and the image obtained after just one iteration of the generative process. In (b) we demostrate the smoothing effect by looking at the marginal distributions and in (c) we compute the one dimesional distribution for $y = -0.67,0.74,0.04$.}
    \label{fig:2d_dog}
\end{figure}

Finally, we illustrate the proposed methodology on a grayscale image corrupted by additive Gaussian noise. We consider the image of a dog shown in Figure~\ref{fig:2d_dog}(a), of size $256 \times 180$ pixels. This example also highlights that the proposed numerical scheme is applicable to non-square geometries. A Gaussian noise perturbation is applied to the original image, after which we employ the generative filtering procedure, as illustrated in the top middle column of Figure~\ref{fig:2d_dog}. The spatial domain is set to $[-L,L]^2$ with $L=8$, and we consider $N_x = 256$ and $N_y = 180$ grid points. For the model parameters, we set $\alpha = 1$ and perform one iteration of the generative process. The results are reported in the right column of Figure~\ref{fig:2d_dog}, where we display the original, noisy, and filtered images. We also present the marginal distributions for each case, highlighting the filtering effect of the generative process. Additionally, we show slices of the distribution for selected values $y = -0.67, 0.04, 0.74$, where the smoothing effect is clearly visible through the attenuation of peaks. Overall, the method produces a smoothed version of the noisy image. While sharp edges are not fully recovered, the procedure effectively reduces the granularity introduced by the Gaussian noise.

\subsection{Particle-Based Generative Dynamics}

In this section, we demonstrate the generative capability of the particle dynamics for $\alpha = 0$, using the scheme in Equation~\eqref{eq:sde_back_upd}. We consider $N_p = 10^6$ particles in two dimensions, initially distributed according to the ring-shaped density
\begin{figure}
\centering

\begin{minipage}{0.35\textwidth}
\centering
\includegraphics[width=\textwidth]{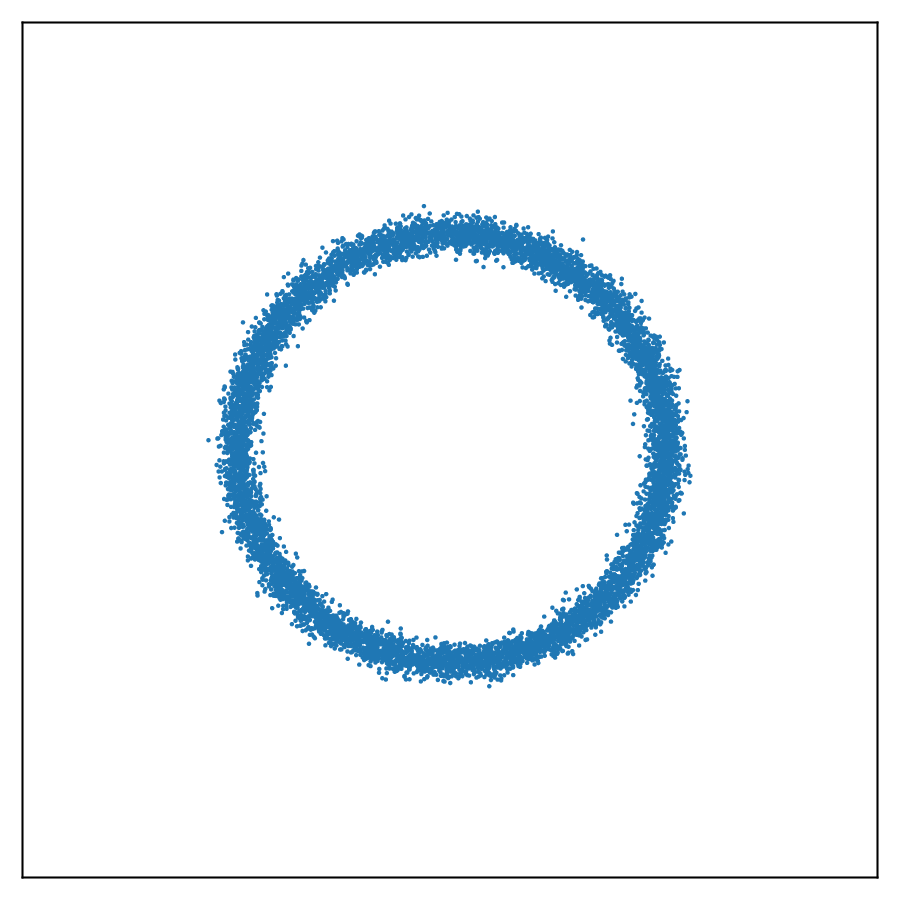}
\end{minipage}
\hspace{0.05cm}
\begin{minipage}{0.35\textwidth}
\centering
\includegraphics[width=\textwidth]{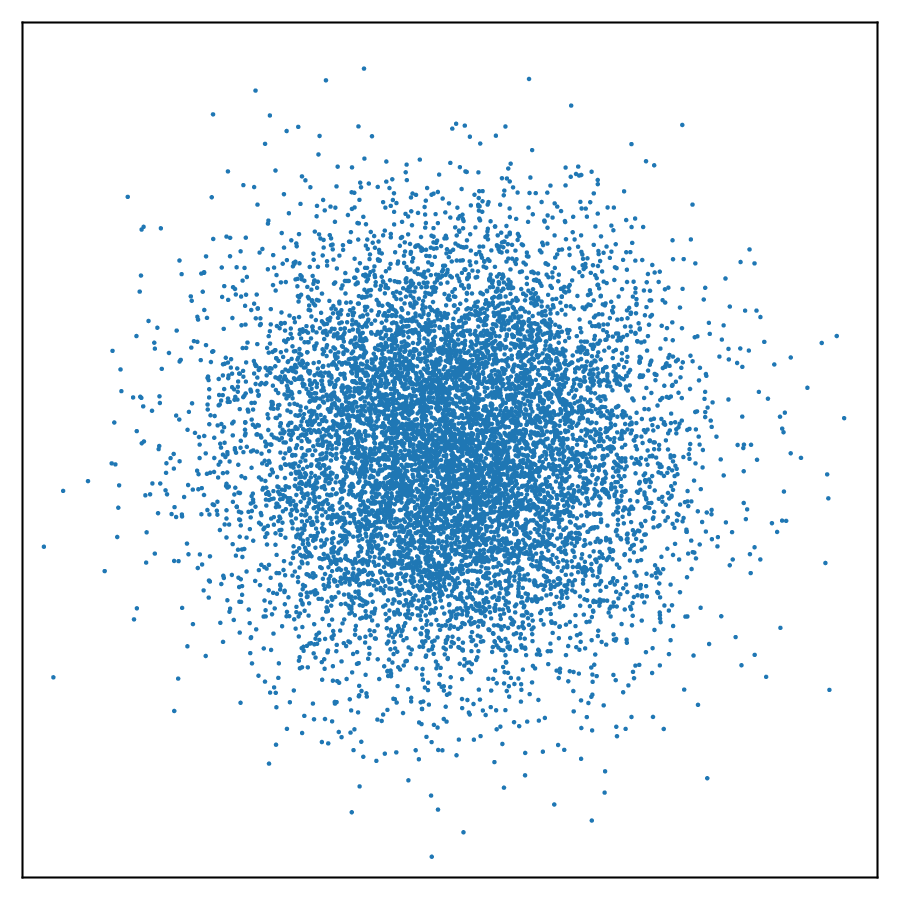}
\end{minipage}

\begin{minipage}{0.4\textwidth}
\centering
\includegraphics[width=\textwidth]{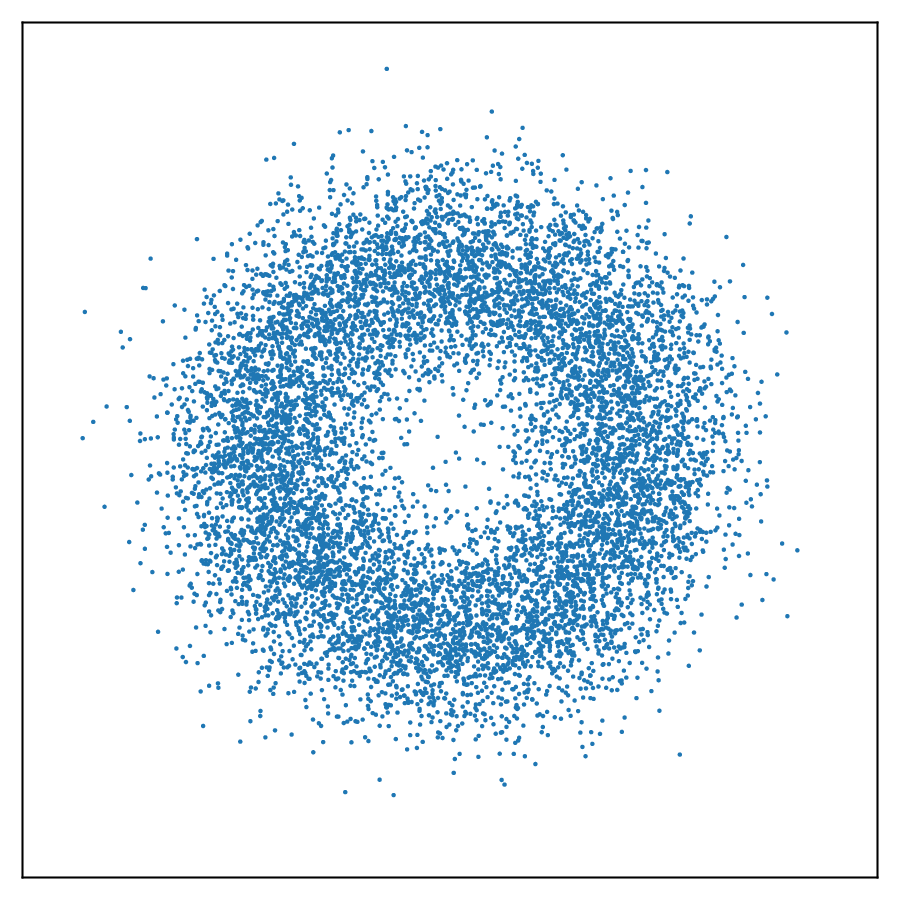}
\end{minipage}

\begin{minipage}{0.4\textwidth}
\centering
\includegraphics[width=\textwidth]{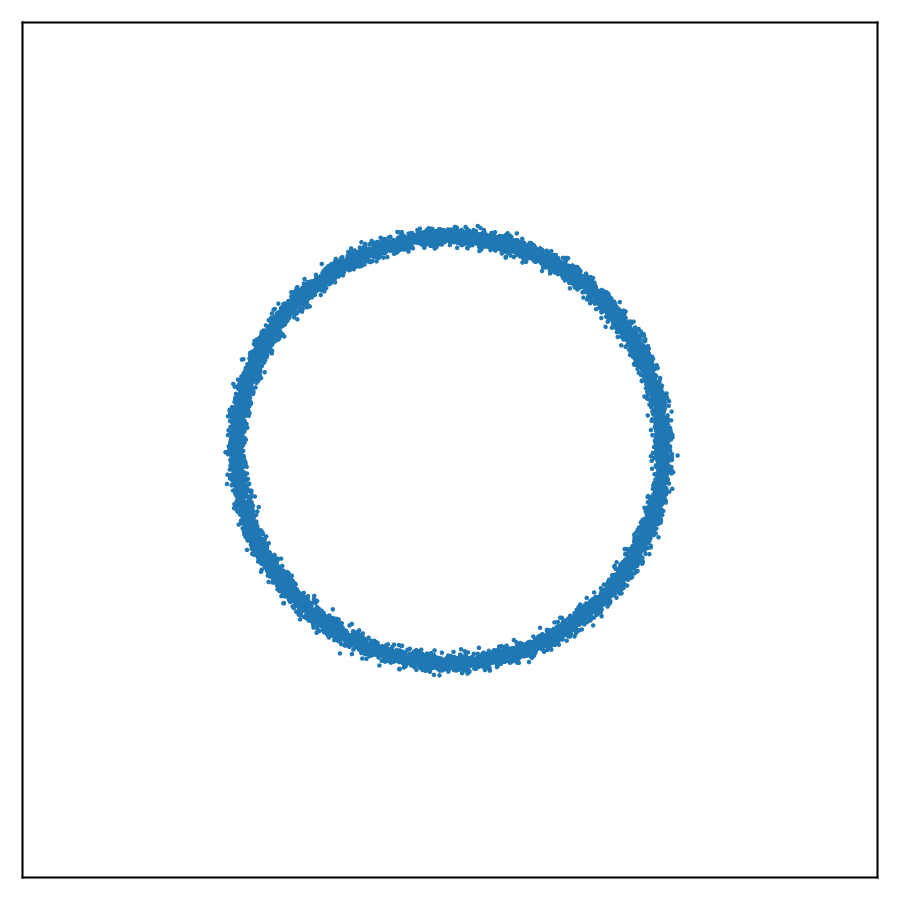}
\end{minipage}

\caption{Particles initially sampled from a ring distribution (top row) spread under the forward dynamics. The backward process (last columns) progressively reconstructs the ring structure (bottom row).}
\label{fig:particle_2d}
\end{figure}

\begin{figure}
\centering

\begin{subfigure}{0.48\textwidth}
\centering
\includegraphics[width=0.8\textwidth]{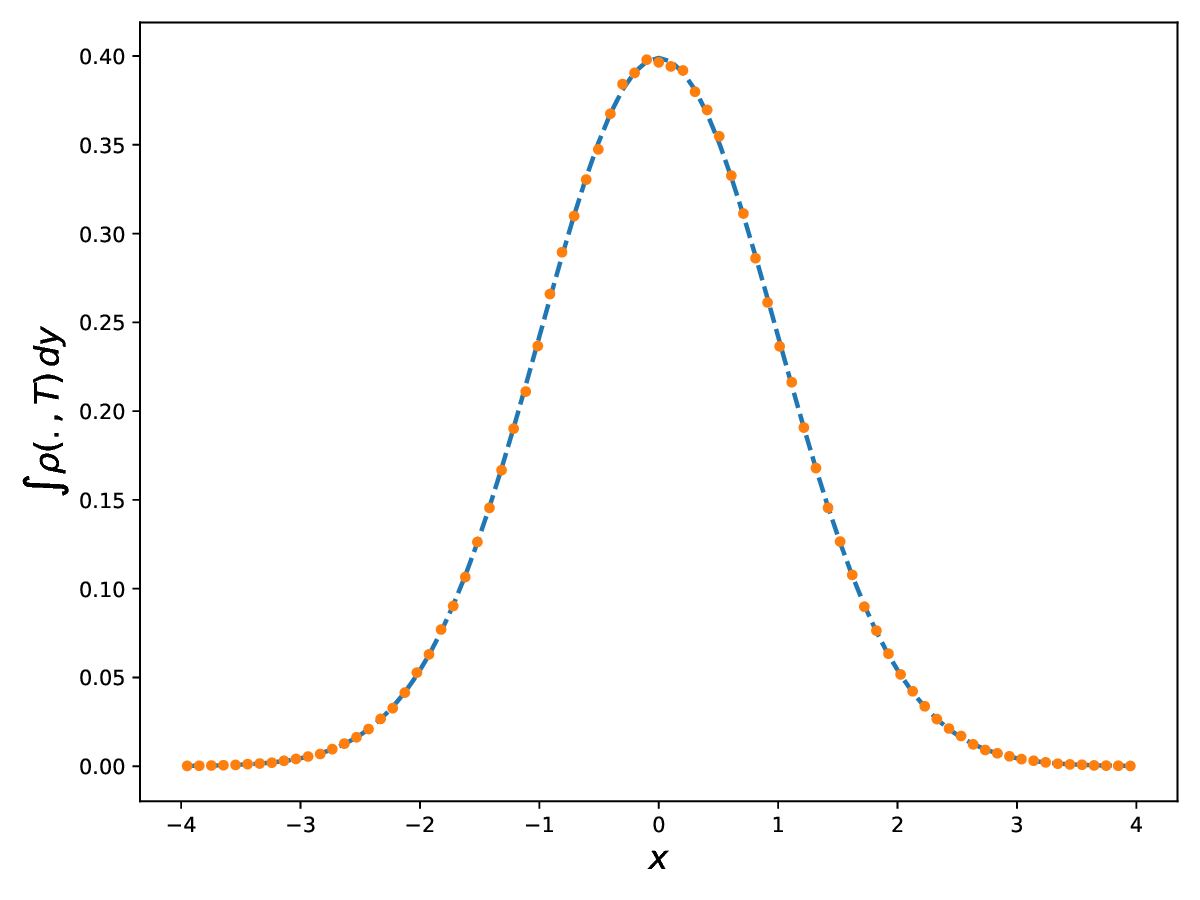}
\caption{}
\end{subfigure}
\hfill
\begin{subfigure}{0.48\textwidth}
\centering
\includegraphics[width=0.8\textwidth]{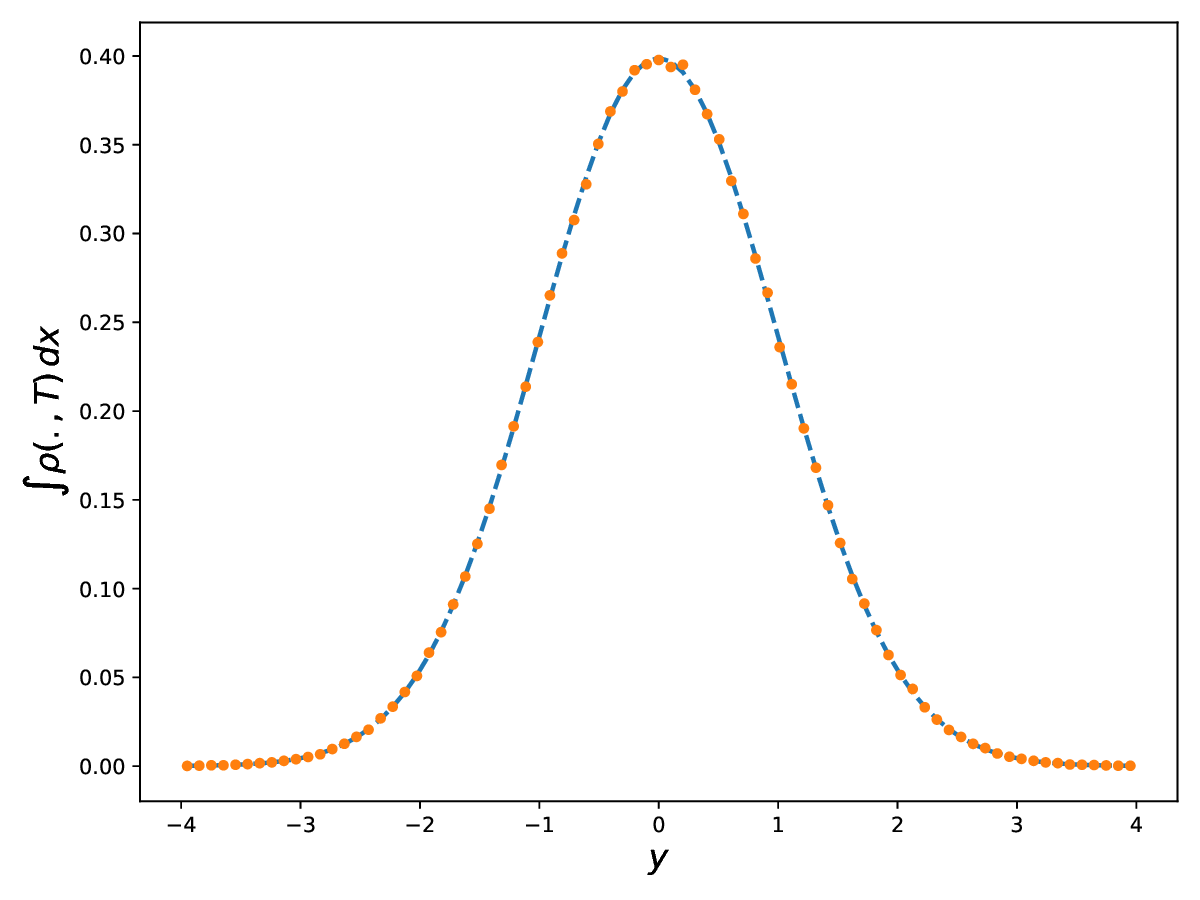}
\caption{}
\end{subfigure}

\vspace{0.3cm}

\begin{subfigure}{0.48\textwidth}
\centering
\includegraphics[width=0.8\textwidth]{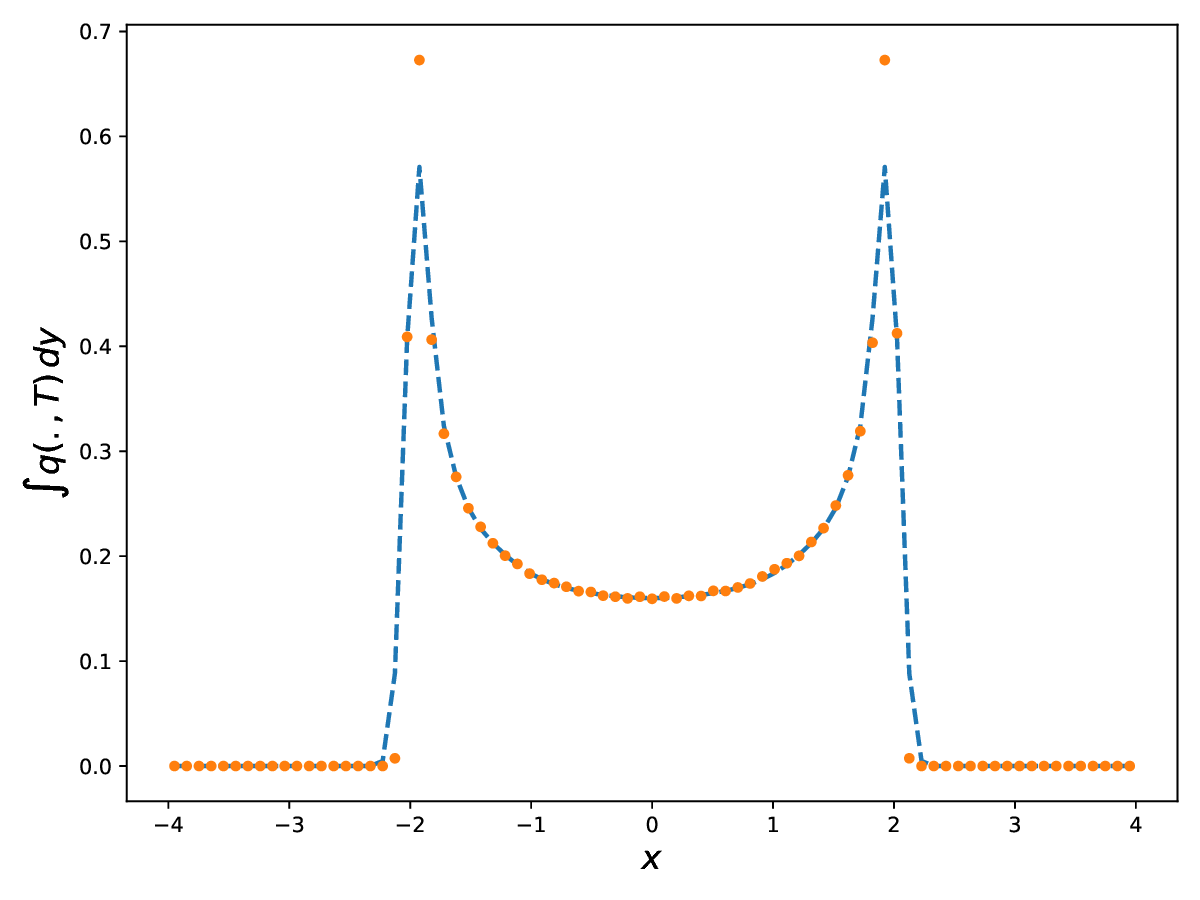}
\caption{}
\end{subfigure}
\hfill
\begin{subfigure}{0.48\textwidth}
\centering
\includegraphics[width=0.8\textwidth]{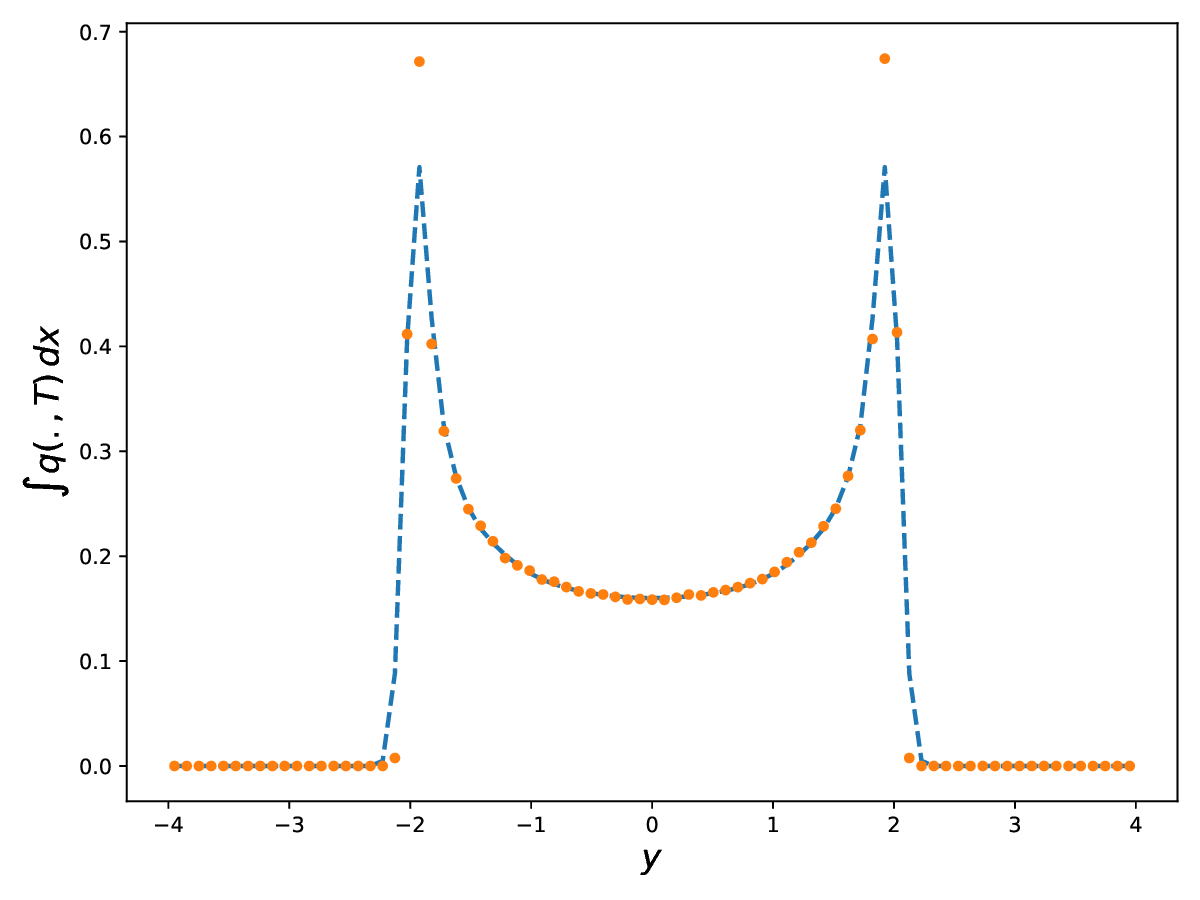}
\caption{}
\end{subfigure}

\caption{Marginal distributions obtained at the end of the forward and reverse processes together with the corresponding analytical marginals using a system of $N_p = 10^6$ particles through the scheme shown in Equation~\eqref{eq:sde_back_upd} with a final time $T_f = 3$ and a step size $\Delta t = 10^{-2}$. We observe that the marginals distributions obtained from the forward and backward processes are in good agreement with the corresponding analytical marginal distributions, confirming the generative capacity of the particle dynamics. We report the $L1-$norm errors: \textbf{Forward process:} $E_x = 0.006$ and $E_y = 0.006$, \textbf{Backward process:} $E_x = 0.048$ and $E_y = 0.048$.}
\label{fig:2x2}
\end{figure}

\begin{figure}
\centering
\includegraphics[width=0.8\textwidth]{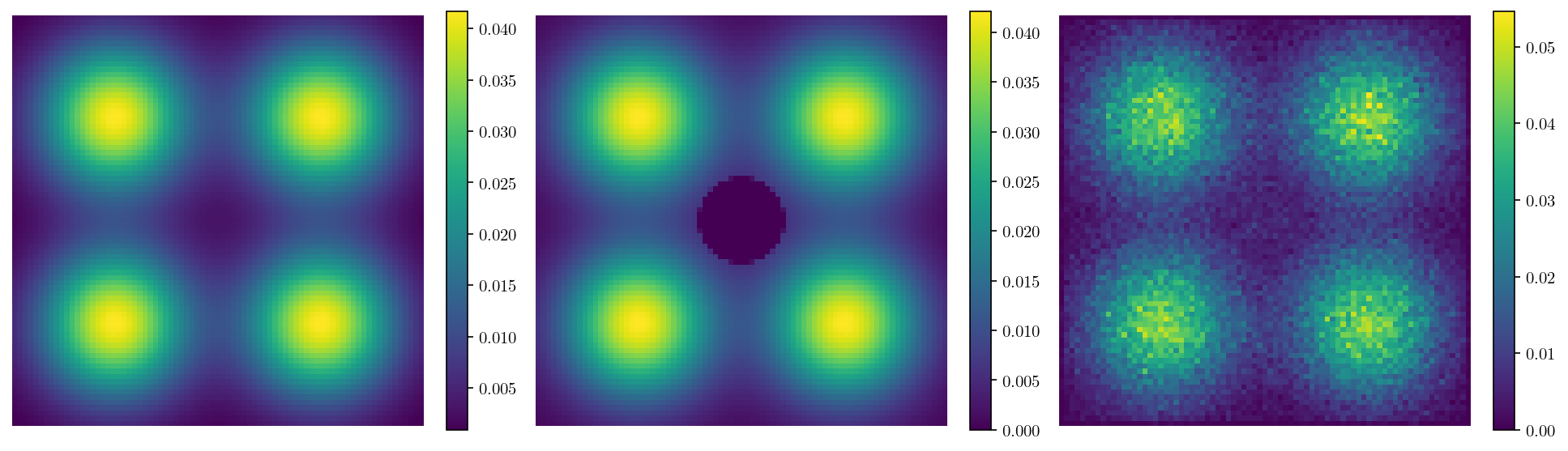}
\caption{By modifying the prior distribution, it is possible to reconstruct regions of a density where data is missing. We demonstrate the original distribution, the incomplete distribution where there is missing data around the center of the distribution, and the reconstructed distribution with a histogram-based approximation with $80$ bins in each spatial direction and $N_p = 10^5$ particles. The reconstructed distribution corresponds to the optimal value of $A$ that minimizes the reconstruction error.}
\label{fig:2d_recon}
\end{figure}

\begin{figure}
\centering
\includegraphics[width=0.7\textwidth]{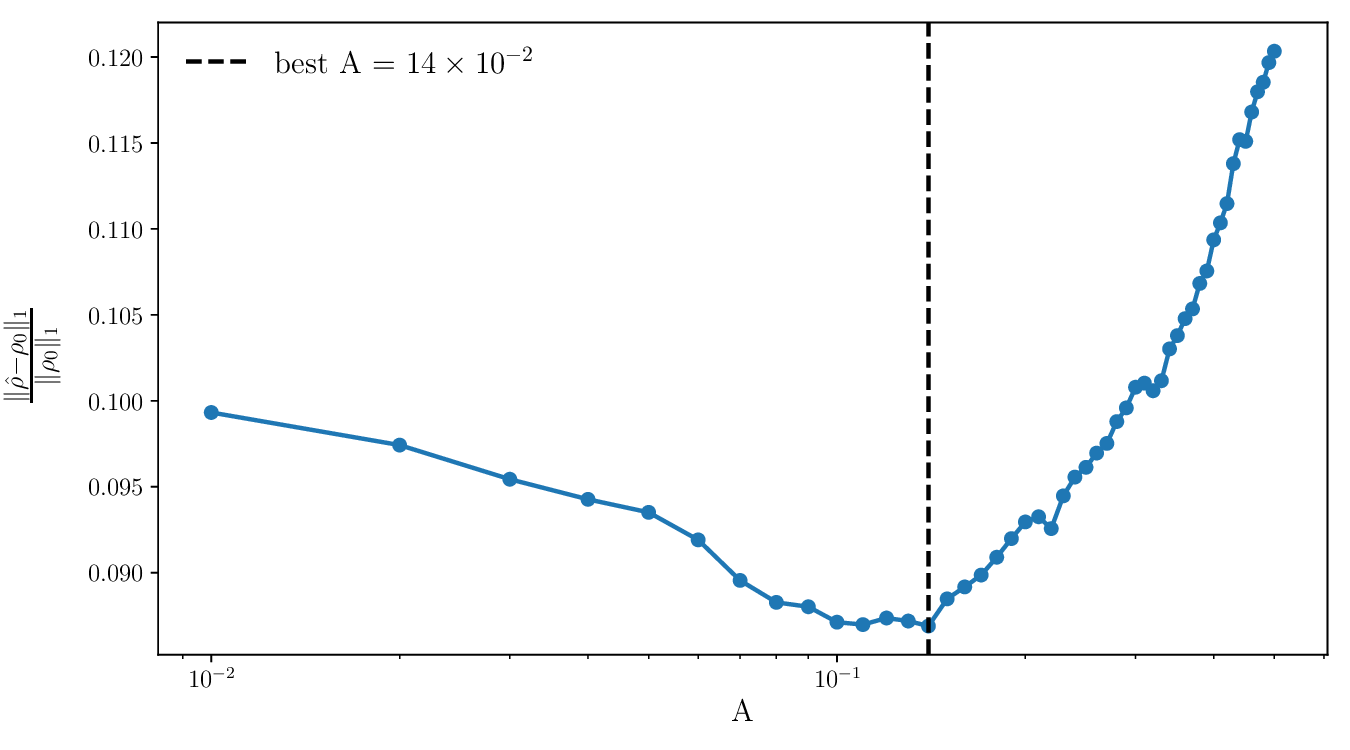}
\caption{For different values of $A$, we evaluate the relative $L^1$ error between the reconstructed and analytical distributions. We observe that there exists an optimal choice of $A$ for which the reconstruction error is minimized and the original distribution is best recovered.}
\label{fig:2d_recon_error}
\end{figure}

\begin{equation}
    \rho_0(x,y) = C \exp \left\{ -\frac{1}{2} \left(\frac{r(x,y)-R}{\sigma_r}\right)^2 \right\},
\end{equation}

with $r(x,y)=\sqrt{x^2+y^2}$, $R=2$, and $\sigma_r = 0.08$. The reconstruction is performed via a histogram with $80$ bins per spatial direction. The simulation is run up to $T_f = 3$ with time step $\Delta t = 10^{-2}$. As shown in Figure~\ref{fig:particle_2d} we follow the trajectory of $10^3$ randomly chosen particles through out the entire diffusion process. The forward process diffuses the initial ring, smoothing the distribution until it becomes approximately Gaussian. The backward dynamics then reverse this evolution, progressively concentrating the particles along the original circular structure. An intermediate configuration illustrates this transition, highlighting the reorganization of mass from a diffuse state to a sharply localized ring.

Furthermore, we compute the marginals distributions at $\rho(.,T)$ and $q(.,T)$, i.e. at the end of the forward and backward process respectively, and compare them to the analytical marginal distributions as shown in Figure~\ref{fig:2x2}.

\subsection{Particle-Based Reconstruction of Missing Data}

By modifying the prior distribution, it is possible to reconstruct regions of a density where data is missing. Let $\rho(x)$ be a given distribution and let $I_{\text{gap}} \subset \Omega$ denote a region where the value of $\rho(x)$ is unknown (or set to zero due to lack of data). We define the modified prior

\begin{equation}
\hat{\rho}(x) =
\begin{cases}
\rho(x), & x \notin I_{\text{gap}}, \\[4pt]
\gamma, & x \in I_{\text{gap}},
\end{cases}
\end{equation}

where $\gamma = A \max_x \rho(x)$ and $A > 0$ is a tunable parameter. By varying $A$, we effectively control the amount of prior mass introduced in the missing region, which directly influences the backward diffusion process. In particular, there exists an optimal value of $A$ for which the reconstructed distribution best approximates the original one. To illustrate this effect, we consider a two-dimensional setting with initial distribution

\begin{equation}
f_0(x,y) = \sum_{\boldsymbol{\mu}\in \mathcal{M}} 
\exp\!\left(-\| (x,y) - \boldsymbol{\mu} \|^2\right),
\quad 
\mathcal{M} = \{(2,2),(2,-2),(-2,2),(-2,-2)\},
\end{equation}

and define the missing region as
\[
I_{\text{gap}} = \left\{(x,y)\in\Omega:\sqrt{x^2+y^2}<0.85\right\}.
\]

We sample $N_p = 10^5$ particles and evolve the system up to final time $T_f = 3$ using a time step $\Delta t = 10^{-2}$. The diffusion-based reconstruction is then performed for different values of $A$. In Figure~\ref{fig:2d_recon}, we show the original distribution, the incomplete prior, and the reconstructed distribution corresponding to the optimal value of $A$. The reconstruction is obtained using a histogram-based approximation with $80$ bins in each spatial direction. Finally, in Figure~\ref{fig:2d_recon_error}, we report the relative $L^1$ error between the reconstructed and the analytical distributions for different values of $A$. The results indicate the existence of an optimal value of $A$ that minimizes the reconstruction error, highlighting the importance of the prior distribution in recovering missing regions.

\section{Conclusions}

In this work, we introduced an alternative generative framework based on a nonlinear Fokker--Planck equation with a superlinear drift term. Starting from a nonlinear interacting particle system, we derived its macroscopic description through the mean-field limit and constructed a forward--backward generative process extending the classical score-based diffusion framework to a broader class of nonlinear dynamics. A distinctive feature of the proposed model is the occurrence of finite-time condensation. We proved that, for suitable values of the model parameters and sufficiently large initial mass, the solution loses $L^2$ regularity in finite time, leading to the attainment of the asymptotic condensed state within a finite time horizon.

Based on the proposed forward dynamics, we derived a stabilized reverse-time diffusion process capable of reconstructing the initial distribution from the stationary state of the forward evolution. Furthermore, we introduced numerical discretizations for both, the forward and backward equations that accurately capture the asymptotic behavior of the continuous model. It also  successfully reconstructs  initial distribution in both one- and two-dimensional settings, thereby preserving the generative properties of the proposed framework at the numerical level.

The proposed methodology has been illustrated through a series of computational experiments. Beyond validating its generative capability, we showed that the iterative application of the forward and backward dynamics can be used   as a filtering mechanism for noisy distributions and images. Moreover, by exploiting the particle formulation of the reverse process, we demonstrated that missing regions of a distribution can be reconstructed.  

Overall, the present work shows that the generative paradigm is not restricted to the classical score-based diffusion model, but can be naturally extended to a broader class of nonlinear evolution equations. We believe that the PDE viewpoint adopted here provides a complementary framework for studying the analytical properties of generative models and for constructing alternative forward and reverse dynamics beyond the classical Ornstein--Uhlenbeck process. Future research will focus on the design of numerical schemes capable of preserving the generative properties of the proposed framework in the presence of singularity formation, as well as  the development of a  theoretical framework for nonlinear evolution equations in this context.

\subsection*{Acknowledgements}
This work has been supported by the European Union’s Horizon Europe research and innovation programme under the Marie Sklodowska-Curie Doctoral Network Datahyking (Grant No. 101072546) as well as by the Deutsche Forschungsgemeinschaft (DFG, German Research Foundation) under HE5386/30-1  Effiziente Ermittlung von Temperaturfeldern mithilfe von Physics-Informed Neural Networks zur Modellierung des thermo-elastischen Verhaltens von Werkzeugmaschinen (537928890) and HE5386/33-1 Control of Interacting Particle Systems, and Their Mean-Field, and Fluid-Dynamic Limits (560288187).

\end{document}